\documentclass{amsart}
\usepackage[utf8]{inputenc}
\usepackage{fourier}
\usepackage{amsmath}
\usepackage{comment}
\usepackage{amssymb}
\usepackage{hyperref}
\hypersetup{
    colorlinks,
    linkcolor={red!70!black},
    citecolor={blue!70!black},
    urlcolor={blue!90!black}
}
\usepackage[english]{babel}

\newcommand{\from}{\colon}
\newcommand{\zero}{\mathbf{0}}
\newcommand{\ud}{\,\mathrm{d}}
\DeclareMathOperator{\TV}{TV}

\DeclareMathOperator{\conv}{conv}
\DeclareMathOperator{\supp}{supp}
\DeclareMathOperator{\diam}{diam}
\DeclareMathOperator{\Lip}{Lip}
\DeclareMathOperator{\id}{id}
\DeclareMathOperator{\vol}{vol}
\DeclareMathOperator{\clos}{clos}
\DeclareMathOperator{\inter}{int}

\DeclareMathOperator{\Gr}{Gr}
\DeclareMathOperator{\FE}{FE}
\DeclareMathOperator{\mass}{mass}
\DeclareMathOperator{\Sym}{Sym}

\newcommand{\cF}{\mathcal{F}}
\newcommand{\cM}{\mathcal{M}}
\newcommand{\cV}{\mathcal{V}}
\newcommand{\R}{\mathbb{R}}
\newcommand{\Q}{\mathbb{Q}}
\newcommand{\N}{\mathbb{N}}
\newcommand{\Z}{\mathbb{Z}}
\newcommand{\bI}{\mathbf{I}}
\newcommand{\bR}{\mathbf{R}}
\newcommand{\one}{\mathbf{1}}
\newcommand{\cH}{\mathcal{H}}
\newcommand{\Cpoly}{C^{\textup{poly}}}

\newtheorem{thm}{Theorem}[section]
\newtheorem{prop}[thm]{Proposition}
\newtheorem{lemma}[thm]{Lemma}

\theoremstyle{remark}
\newtheorem{defn}[thm]{Definition}

\newtheorem{remark}[thm]{Remark}

\usepackage{pgfornament}

\newcommand{\OGr}{\widetilde{\Gr}}
\newcommand{\rstr}{\mathbin{\vrule height 1.4ex depth 0pt width
0.14ex\vrule height 0.14ex depth 0pt width 1ex}}

\newcommand{\smrstr}{\mathbin{\vrule height 1ex depth 0pt width
0.1ex\vrule height 0.1ex depth 0pt width .7ex}}

\title{Construction of fillings with prescribed Gaussian image and applications}

\author{Antonio De Rosa}
\address{Department of Decision Sciences and BIDSA, Bocconi University, Milano, Italy}
\email{antonio.derosa@unibocconi.it}

\author{Yucong Lei}
\address{University of Michigan}
\email{leiyc@umich.edu}

\author{Robert Young}
\address{Courant Institute of Mathematical Sciences, New York University}
\email{ryoung@cims.nyu.edu}

\thanks{Antonio De Rosa was funded by the European Union: the European Research Council (ERC), through StG ``ANGEVA'', project number: 101076411. Views and opinions expressed are however those of the authors only and do not necessarily reflect those of the European Union or the European Research Council. Neither the European Union nor the granting authority can be held responsible for them. Robert Young's work was supported by the National Science Foundation under Grant No.\ DMS-2005609.}
\date{\today}

\begin{document}

\begin{abstract}
  We construct $d$--dimensional polyhedral chains such that the distribution of tangent planes is close to a prescribed measure on the Grassmannian and the chains are either cycles (if the barycenter of the prescribed measure, considered as a measure on $\bigwedge^d \R^n$, is $0$) or their boundary is the boundary of a unit $d$--cube (if the barycenter of the prescribed measure is a simple $d$--vector). 
  Such fillings were first proved to exist by Burago and Ivanov \cite{BurIva}; our work gives an explicit construction, which is also flexible to generalizations. For instance, in the case that the measure on the Grassmannian is supported on the set of positively oriented $d$--planes, we can construct fillings that are Lipschitz multigraphs. We apply this construction to prove the surprising fact that, for anisotropic integrands, polyconvexity is equivalent to quasiconvexity of the associated $Q$-integrands (that is, ellipticity for Lipschitz multigraphs) and to show that strict polyconvexity is necessary for the atomic condition to hold. 
\end{abstract}

\maketitle

\section{Introduction}
Positive functions on the Grassmannian give rise to anisotropic energy functionals on surfaces. That is, for a bounded measurable function $$\Psi\from \OGr(d,n)\to (0,\infty)$$ on the oriented Grassmannian $\OGr(d,n)$, we define the energy of a $d$--dimensional oriented smooth surface $\Sigma \subset \R^n$ with respect to the integrand $\Psi$ as 
$$F_\Psi(\Sigma) = \int_\Sigma \Psi(T_x\Sigma) \ud \cH^d(x),$$
where $T_x\Sigma$ is the oriented tangent plane of $\Sigma$ at $x$ and $\mathcal{H}^d$ denotes the $d$--dimensional Hausdorff measure. We call $F=F_\Psi$ an \emph{anisotropic energy functional}. If 
$\Psi(P)=1$ for all $P\in \OGr(d,n)$, the energy $F_\Psi(\Sigma)$ is the surface area $\mathcal{H}^d(\Sigma)$.

We can extend $F_\Psi$ to the class of rectifiable currents in the natural way. That is, if $A=(E,\theta \omega_P)$ is the rectifiable $d$--current corresponding to a countably $d$--rectifiable set $E\subset \R^n$,  an $\mathcal H^{d}$--measurable, unitary, simple $d$--vector field $\omega_P:E \to \bigwedge^d\R^n$ associated to an orientation $P\from E\to \OGr(d,n)$ of the tangent bundle of $E$, and a multiplicity $\theta\from E\to [0,\infty)$, then 
\begin{equation}\label{energy}
    F_\Psi(A) = \int_E \theta(x)\Psi(P(x)) \ud \cH^d(x).
\end{equation}

Minimizers and critical points of anisotropic energy functionals are key objects of study in geometric measure theory; we refer to \cite{DR} for a survey on the anisotropic minimal surfaces theory, cf. \cite{Almgren68,DePDeRGhi3,DPM,DeRosa,EVA,FIGAL,MooneyYang,SchoenSimonAlm,LSP}.
Their regularity is often related to the ellipticity of $F$, i.e., whether $d$--planes in $\R^n$ are local minimizers of $F$. A $d$--plane $P$ is a local minimizer if and only if the unit $d$-dimensional ball $B\subset P$ minimizes $F$ among all competitors whose boundary is $\partial B$. When stating that $F$ is elliptic, we always need to specify the class of competitors used to test the local minimality of $d$-planes. 

In this paper, we will study how the ellipticity of $F$, and hence the regularity of the minimizers and critical points of $F$, depend on the set of competitors. We are particularly interested in the difference between competitors with multiplicity $1$, such as integral currents or graphs of functions, and competitors with arbitrary multiplicity, such as rectifiable currents with real coefficients or multigraphs.

These differences were studied by Burago and Ivanov in \cite{BurIva}, where they proved the following Theorem \ref{thm:inf-formula} and Theorem \ref{thm:BurIvaNonPoly}. To state these theorems we need to introduce some notation.
Let $\bR^d(\R^n)$ be the set of rectifiable currents with real coefficients and let $\bI^d(\R^n)$ be the set of integral currents. For a cycle $S\in \bI^{d-1}(\R^n)$, we define the filling energy of $S$ with real coefficients by 
$$\FE_F(S;\R) := \inf\left\{F(A) : A\in \bR^d(\R^n), \partial A = S\right\}$$
and with integer coefficients by
$$\FE_F(S;\Z) := \inf\left\{F(A) : A\in \bI^d(\R^n), \partial A = S\right\}.$$
Let $\mathcal{M}(\OGr(d,n))$ be the set of positive Radon measures on $\OGr(d,n)$, and for $P\in \OGr(d,n)$, let $\omega_P\in \bigwedge^d\R^n$ be the unit simple $d$--vector corresponding to $P$. We will frequently view $\OGr(d,n)$ as a set of unit vectors in $\bigwedge^d \R^n$ by identifying $P$ with $\omega_P$. 
\begin{thm}[{\cite[Thm.~3]{BurIva}}]\label{thm:inf-formula}
  Let $\Psi\from \OGr(d,n)\to (0,\infty)$ be a bounded, measurable function. Let $P_0 \in \OGr(d,n)$ and let $S\in \bI^{d-1}(\R^n)$ be the fundamental class of the boundary of a unit $d$--cube in $P_0$. 
  Then
  \begin{multline}\label{eq:inf-formula}
    \FE_F(S;\R) = \lim_{n\to \infty} n^{-1}\FE_F(nS;\Z) \\= \inf \left\{\int \Psi(P)\ud \mu(P) : \mu\in \mathcal{M}(\OGr(d,n)), \int \omega_P \ud \mu(P) = \omega_{P_0}\right\}.
  \end{multline}
  Consequently, $F_\Psi$ is elliptic for rectifiable currents with real coefficients if and only if \(\Psi\) can be extended to a $1$--homogeneous convex function on $\bigwedge^d\R^n$.
\end{thm}

\begin{thm}[{\cite[Thm.~5]{BurIva}}]\label{thm:BurIvaNonPoly}
  There is a continuous $\Psi \from \OGr(2,4) \to (0,\infty)$ such that $F_\Psi$ is elliptic for piecewise smooth surfaces, but not for $\bR^d(\R^n)$. 
\end{thm}

\begin{defn}\label{def:polycon-general}
  If $\Psi\from \OGr(d,n)\to (0,\infty)$ can be extended to a $1$--homogeneous convex function on $\bigwedge^d\R^n$, we say that $\Psi$ is \emph{polyconvex}. If there is a strictly convex extension (i.e., an extension such that $\Psi(v + w) < \Psi(v) + \Psi(w)$ whenever $v$ and $w$ are linearly independent), we call $\Psi$ \emph{strictly polyconvex}.
\end{defn}

We deduce from Definition \ref{def:polycon-general} that, if $\Psi\from \OGr(d,n)\to (0,\infty)$ is strictly  polyconvex, then $\int \Psi(P)\ud \mu(P)\geq  \Psi(P_0)$ for every $P_0 \in \OGr(d,n)$ and every $\mu\in \mathcal{M}(\OGr(d,n))$ satisfying $\int \omega_P \ud \mu(P) = \omega_{P_0}$, with equality if and only if $\mu=\delta_{P_0}$.

\begin{remark}\label{observ}
We observe that, equivalently to Definition \ref{def:polycon-general}, $\Psi$ is (strictly) polyconvex if there is a $1$--homogeneous (strictly) convex function $f \from \R^{{n \choose d}} \to [0,\infty)$ such that $\Psi(P) = f(m_P)$ for all $P\in \OGr(d,n)$, where $m_P\in \R^{{n \choose d}}$ is the vector of $d\times d$ minors of an $n\times d$ matrix $M_P$ whose columns $M_P^j$ are a positively-oriented orthonormal basis of $P$.
    Hence the notion of polyconvexity for $\Psi$ is consistent with the classical notion of polyconvexity for functions defined on $\R^{(n-d)\times d}$, see \cite[Definition 4.2]{Mull}, which we recall later in Definition \ref{polyconvexitygraph}. 
\end{remark}

Burago and Ivanov's proof of Theorem~\ref{thm:inf-formula} is based on a nonconstructive proof of the existence of surfaces with prescribed tangent planes. In this paper, we give an explicit construction of such surfaces (Theorem \ref{thm:construct}) which generalizes to other classes of surfaces. For instance, in the case the tangent planes are assumed to be positively oriented $d$--planes, we can construct Lipschitz multigraphs. We use this construction in Theorems~\ref{thm:poly-multi} and \ref{thm:poly-multi-conseq} to extend Theorems~\ref{thm:inf-formula} and \ref{thm:BurIvaNonPoly} to Lipschitz multigraphs. This lets us show (Theorem~\ref{qcvspc}) that polyconvexity of an anisotropic integrand and quasiconvexity of the associated $Q$--integrands are equivalent. Finally, we  study the atomic condition that was introduced in \cite[Definition 1.1]{DePDeRGhi} and we prove that it implies strict polyconvexity, see Theorem \ref{thm:main-atomic2}. 

\vspace{0.1cm}

\subsection{Constructing surfaces with prescribed tangent planes}
We first recall the notion of a Lipschitz multigraph. A multigraph is the graph of a $Q$--valued function for some $Q>0$. To define such functions, let $\delta_{x_i}$ be the Dirac mass in $x_i\in \mathbb R^{n-d}$ and let
\begin{equation*}
\mathcal{A}_Q(\mathbb{R}^{n-d}) :=\left\{\sum_{i=1}^Q\delta_{x_i}\,:\,x_i\in\mathbb R^{n-d}\;\textrm{for every  }i=1,\ldots,Q\right\},
\end{equation*}
equipped with the Wasserstein distance.
We can identify $\mathcal{A}_Q(\R^{n-d})$ with the configuration space $(\R^{n-d})^Q/\Sym(Q)$.
A \emph{$Q$--valued function} is then a map from an open set $U\subset \R^d$ to $\mathcal{A}_Q(\mathbb{R}^{n-d})$; a \emph{Lipschitz $Q$--valued function} is a Lipschitz map from $U$ to $\mathcal{A}_Q(\mathbb{R}^{n-d})$.

A \emph{multigraph} is the graph $\Lambda_f$ of a $Q$--valued function $f\from U\to \mathcal{A}_Q(\mathbb{R}^{n-d})$, defined by
$$\Lambda_f = \{(x,y)\in U\times \R^{n-d}, y \in \supp f(x)\};$$
note that for all $x$, $\supp f(x)$ has cardinality at most $Q$. A \emph{Lipschitz multigraph} is the graph of a Lipschitz $Q$--valued function, and such a graph naturally supports an integral current $[\Lambda_f]\in \bI^d(\R^n)$; see \cite[Definition 1.10]{DLSp2}.

For every oriented $d$-dimensional rectifiable surface $\Sigma$ in $\R^n$ we denote with $[\Sigma]:=(\Sigma, \omega_{T_x\Sigma})\in \bI^d(\R^n)$ the fundamental class of $\Sigma$, i.e., the multiplicity-one integral current associated to $\Sigma$.
For a rectifiable $d$--current $A=(E, \theta \omega_P)\in \bR^d(\R^n)$, we define \emph{the weighted Gaussian image} $\gamma_A\in \cM(\OGr(d,n))$ of $A$ as the pushforward $P_*(\theta \cH^d \rstr E)$ of the measure $\theta \cH^d \rstr E$ under the orientation field $P:E \to \OGr(d,n)$.  For example, if  $A=\sum_{i=1}^k a_i [\Delta_i]$ is a polyhedral $d$--chain in $\R^n$ such that $a_i\geq 0$, the $\Delta_i$'s are oriented $d$--dimensional polyhedra with corresponding orientations $P_i\in \OGr(d,n)$, and $\cH^d(\Delta_i\cap \Delta_j)=0$ for all $i\ne j$, then $\gamma_A=\sum_{i=1}^k a_i \cH^d(\Delta_i) \delta_{P_i}$.

Let $p\from \OGr(d,n)\to \Gr(d,n)$ denote the map that forgets orientation. For every rectifiable $d$--current $A=(E, \theta \omega_P)\in \bR^d(\R^n)$ we will denote the corresponding varifold by $\cV(A):=(\cH^d \rstr E )\otimes \theta \delta_{p(P(x))}$.

Our first main theorem constructs surfaces with prescribed tangent planes:
\begin{thm}\label{thm:construct}
  Let $P_0 \in \OGr(d,n)$, $\mu \in \mathcal{M}(\OGr(d,n))$, and $D$ be an oriented unit $d$-cube in $P_0$. 
  \begin{itemize}
  \item If $\int \omega_P \ud \mu(P) = 0$, then there is a sequence $\Sigma_i$ of polyhedral chains such that $\partial \Sigma_i = 0$ for all $i$, $\gamma_{\Sigma_i}$ converges weakly to $\mu$, $\supp \Sigma_i$ converges to $[0,1]^n$ in the Hausdorff metric, and the corresponding varifolds $\cV(\Sigma_i)$ converge to the varifold $(\cH^n\rstr [0,1]^n)\otimes p_*\mu$, where $p_*\mu$ denotes the pushforward of the measure $\mu$ under the map $p$.
  \item 
  If $\int \omega_P \ud \mu(P) = \omega_{P_0}$, then there is a sequence $\Sigma_i$ of polyhedral chains such that $\partial \Sigma_i =\partial [D]$ for all $i$, $\gamma_{\Sigma_i}$ converges weakly to $\mu$, $\supp \Sigma_i$ converges to $D$ in the Hausdorff metric, and the corresponding varifolds $\cV(\Sigma_i)$ converge to the varifold $(\cH^d\rstr D)\otimes p_*\mu$.
  \item    
  If $\int \omega_P \ud \mu(P) = \omega_{P_0}$ and $\supp \mu \subset \OGr^+(d,n)$, where $\OGr^+(d,n)$ denotes the subset of positively oriented $d$-planes with respect to $P_0$, then there is a sequence of currents $\Sigma_i$, each corresponding to a multiple of a Lipschitz multigraph over $P_0$, such that $\partial \Sigma_i =\partial [D]$ for all $i$, $\gamma_{\Sigma_i}$ converges weakly to $\mu$, $\supp \Sigma_i$ converges to $D$ in the Hausdorff metric, and the corresponding varifolds $\cV(\Sigma_i)$ converge to the varifold $(\cH^d\rstr D)\otimes p_*\mu$.
  \end{itemize}
\end{thm}

In the following, we prove two main consequences of Theorem~\ref{thm:construct}. 

\subsection{Equivalence of polyconvexity and ellipticity for Lipschitz multigraphs} As a first application of Theorem~\ref{thm:construct}, we generalize Burago and Ivanov's Theorem \ref{thm:inf-formula} to multigraphs. To state this generalization, we need to define energy and ellipticity for multigraphs.

Given $\psi\from \R^{(n-d)\times d}\to [0,\infty)$ and $Q\in \N$, we define the $\psi$--energy of a $Q$--valued function as follows. For $\{X_i\}_{i=1}^Q\subset \R^{(n-d)\times d}$, let $\bar{\psi}\from (\R^{(n-d)\times d})^Q\to [0,\infty)$ be defined as 
  \begin{equation}\label{integ}
      \bar{\psi}(X_1, \ldots, X_Q):=\sum_{i=1}^Q \psi(X_i).
  \end{equation}
  Clearly the definition of $\bar{\psi}$ depends on $Q$, but we will always omit this dependence for notation simplicity.
  Since $\bar{\psi}$ is invariant under permutations of $\{1,\ldots,Q\}$, it is a $Q$--integrand in the sense of \cite[Section 0.1]{DeLellisFocardiSpadaro}. So we can define
  \begin{equation}\label{functional}
    F_\psi(f)=\int_B \bar{\psi}\big(Df(x)\big)dx
  \end{equation}
  for any weakly differentiable $Q$--valued map $f\from B\subset \R^d \to \mathcal{A}_Q(\mathbb{R}^{n-d})$.

We observe below that $F_\psi(f)$ corresponds closely to the energy of $[\Lambda_f]$, cf. \cite[Equations 2.5-2.7]{DRT}, \cite[Corollary 1.11]{DLSp2} and \cite[Section 5]{T3}. To this aim we introduce the following definition.
\begin{defn}\label{def:matrix-polycon}
  For a matrix $X\in \R^{(n-d)\times d}$, let
  $$M(X):=  \begin{pmatrix} I_d\\ X \end{pmatrix}\in \R^{n\times d}$$
  and let $\bigwedge M(X) = w_1(X)\wedge\dots\wedge w_d(X)$, where $w_1(X),\dots, w_d(X)$ are the columns of $M(X)$. This gives a map $\bigwedge M\from \R^{(n-d)\times d}\to \bigwedge^d\R^n$, and one can calculate that the coordinates of $\bigwedge M(X)$ in the standard basis of $\bigwedge^d\R^n$ are the minors of $X$ of any order, which in turn coincide with the $d\times d$ minors of $M(X)$.
\end{defn}
If $\Psi\from \bigwedge^d\R^n\to [0,\infty)$ is a $1$--homogeneous function, we can think of the restriction of $\Psi$ to the unit simple $d$--vectors as an energy on $\OGr(d,n)$. If we let $\psi\from \R^{(n-d)\times d}\to [0,\infty)$, $\psi(X) = \Psi(\bigwedge M(X))$
for all $X$, then by the area formula for $Q$--valued graphs \cite[Corollary 1.11]{DLSp2} we have $F_\psi(f) = F_{\Psi}([\Lambda_f])$ for any Lipschitz $Q$--valued function $f\from U\subset \R^d\to \mathcal{A}_Q(\mathbb{R}^{n-d})$, as we claimed. 

We now recall the classical notion of polyconvexity for an integrand $\psi\from \R^{(n-d)\times d}\to [0,\infty)$, see \cite[Definition 4.2]{Mull}. As we observed in Remark \ref{observ} and from the discussion above, this is consistent with the notion of polyconvexity for integrands on the oriented Grassmannian in Definition \ref{def:polycon-general}.
\begin{defn}\label{polyconvexitygraph}
  Let $(\bigwedge^d \R^n)^+\subset \bigwedge^d\R^n$ be the subset of $d$--vectors whose projection to $\R^d\subset \R^n$ is positively oriented.   
  We say that $\psi\from \R^{(n-d)\times d}\to [0,\infty)$ is \emph{polyconvex} if there is a $1$--homogeneous convex function $\Psi\from (\bigwedge^d\R^n)^+\to \R$ such that $\psi = \Psi \circ \bigwedge M$, i.e., if $\psi$ can be written as a convex function of the minors of $X$.
\end{defn}

Under some additional assumptions, it can be shown that polyconvexity allows to obtain good regularity properties for stationary graphs \cite{T1,T2}. We will show the following analogue of Theorem~\ref{thm:inf-formula}.
\begin{thm}\label{thm:poly-multi}
 Let $\psi\in C^0(\R^{(n-d)\times d},[0,\infty))$. Then $F_\psi$ is elliptic on Lipschitz multivalued functions if and only if $\psi$ is polyconvex.
\end{thm}
We remark that Theorem~\ref{thm:poly-multi} could be phrased in terms of the quasiconvexity of the $Q$-integrand $\bar{\psi}$, as defined by De Lellis, Focardi and Spadaro in \cite[Definition 0.1]{DeLellisFocardiSpadaro}:
\begin{defn}
[\cite{DeLellisFocardiSpadaro}]\label{qc}
Let $\bar{\psi}: \left(\R^{(n-d)\times d}\right)^Q \to [0,\infty)$ be a locally bounded $Q$--integrand.
We say that $\bar{\psi}$ is \textit{quasiconvex} if the following holds
for every affine $Q$--valued function
$u (x)= \sum_{j=1}^J Q_j \delta_{a_j + L_j x}$, with $x\in [0,1]^d$, $L_j\in \R^{(n-d)\times d}$, $a_j\in \R^{n-d}$, and $Q= \sum_{j=1}^J Q_j$.
Given any collection of Lipschitz $Q_j$--valued functions $f^j\from [0,1]^d\to \mathcal{A}_{Q_j}(\mathbb{R}^{n-d})$
with $f^j|_{\partial [0,1]^d}(x) = Q_j \delta_{a_j +L_jx}$, we have the inequality
$$
\bar{\psi}\big(Du(0)\big)\leq \int_{[0,1]^d} \bar{\psi}\big(Df^1(x), \ldots, Df^J(x)\big)dx.
$$
\end{defn}

\begin{remark}
    We point out that the class of $Q$--integrands considered by De Lellis, Focardi and Spadaro in \cite{DeLellisFocardiSpadaro} is wider than \eqref{integ}, as it contains all symmetric functions of the values of the gradient. We only focus on the $Q$--integrands $\bar{\psi}$ defined in \eqref{integ}, as they recover the anisotropic energy of the multigraphs.  
\end{remark}

\begin{remark}
    Quasiconvexity is a well known condition in the nonlinear analysis literature as it is a sufficient condition for the lower semicontinuity of the energy. We refer the reader to the relevant works by Fonseca \cite{Fons}, Fonseca and M\"uller \cite{FoMu}, De Lellis, Focardi and Spadaro in \cite{DeLellisFocardiSpadaro}.
\end{remark}
It is immediate to see that for $\psi\in C^0(\R^{(n-d)\times d},[0,\infty))$ the energy functional $F_\psi$ is elliptic on Lipschitz multivalued functions if and only if for every $Q\in \N$ the $Q$--integrand $\bar{\psi}$ is quasiconvex. Hence we deduce the following surprising result from Theorem \ref{thm:poly-multi}:

\begin{thm}\label{qcvspc}
 Let $\psi\in C^0(\R^{(n-d)\times d},[0,\infty))$. Then $\psi$ is polyconvex if and only if for every $Q\in \N$ the $Q$--integrand $\bar{\psi}$ is quasiconvex.
\end{thm}

By a result of Šverák \cite{Sverak} and another of Alibert and Dacorogna \cite{AlibertDacorogna}, there is a continuous $\psi\from \R^{2\times 2} \to [0,\infty)$ such that $F_\psi$ is elliptic for Lipschitz graphs but $\psi$ is not polyconvex. Indeed, the ellipticity of $F_\psi$ for Lipschitz graphs is equivalent to the classical notion of quasiconvexity for $\psi$, see \cite[Definition 4.2]{Mull}, which in turn is simply Definition \ref{qc} with $Q=1$. We recall that there were earlier examples of quasiconvex quadratic forms on $\R^{3\times 3}$ that are not polyconvex \cite{Ball, Serre, Terpstra}. However, every quasiconvex quadratic form on $\R^{2\times 2}$ is polyconvex \cite{Serre, Terpstra}, hence the need of the counterexamples in \cite{Sverak, AlibertDacorogna}.
From this discussion, we deduce that Theorem~\ref{thm:poly-multi} implies the following analogue of Theorem~\ref{thm:BurIvaNonPoly} for graphs.
\begin{thm}\label{thm:poly-multi-conseq}
  There is a continuous $\psi \from \R^{2\times 2} \to [0,\infty)$ such that $F_\psi$ is elliptic for Lipschitz graphs but not for Lipschitz multigraphs.
\end{thm}

\begin{remark}
    Theorem \ref{qcvspc} and Theorem \ref{thm:poly-multi-conseq} leave some interesting open questions. For example, is there a $Q\in \N$ such that for every $\psi\in C^0(\R^{(n-d)\times d},[0,\infty))$ it holds that $\psi$ is polyconvex if and only if the $Q$--integrand $\bar{\psi}$ is quasiconvex? We believe that the answer to this question is negative. If $\psi$ is not polyconvex, then by Theorem~\ref{thm:construct}, there is a $Q$--valued graph with the same boundary as the unit cube and strictly smaller energy. If $\psi$ is close to a polyconvex function in a suitable norm, however, such a graph may have to satisfy very strict conditions, which may require $Q$ to be large. 
\end{remark}

\subsection{The atomic condition implies strict polyconvexity} The second main application of Theorem~\ref{thm:construct} is in the study of the atomic condition for anisotropic integrands, which was introduced in \cite[Definition 1.1]{DePDeRGhi} to prove the rectifiability of varifolds with bounded anisotropic first variation \cite[Theorem 1.2]{DePDeRGhi}. Recall that a varifold is a positive Radon measure on $\R^{n}\times \Gr(d,n)$.

Given $\Psi\in C^1(\Gr(d,n),(0,\infty))$, a varifold $V$, and a Borel measurable subset $U\subset \R^n$, it is natural to define the energy of $V$ on $U$ by 
\begin{equation}\label{envar}
F_\Psi(V,U):= \int_{U\times \Gr(d,n)} \Psi(T) d V(x,T),
\end{equation}
the energy of $V$ as $F(V)=F_\Psi(V) := F_\Psi(V,\R^n)$, and the anisotropic first variation with respect to $\Psi$ as 
\[
\delta_\Psi[V](g) := \left.\frac{d}{d\varepsilon}\right|_{\varepsilon = 0}F_\Psi((\id+\varepsilon g)_\sharp(V)), \quad \forall g \in C^1_c(\R^n,\R^n).
\]
Here, for a diffeomorphism $f\from \R^n \to \R^n$, we denote with $f_\sharp(V)$ the image of $V$ under $f$, i.e., the varifold whose action on every $\varphi\in C^0_c(\R^n\times \Gr(d,n))$ is: 
$$\int_{\R^n\times \Gr(d,n)}\varphi(x,T)d(f_\sharp(V))(x,T)=\int_{\R^n\times \Gr(d,n)}\varphi(f(x),d_xf(T))Jf(x,T) d V(x,T),$$
where $Jf(x,T)$ is the $d$--Jacobian determinant of the differential $d_xf$ restricted to $T$, see \cite[Chapter 8]{SIM}. Note that the image of a varifold \(V\) is {\em not} the same as the push-forward of the  Radon measure \(V\) through a map $\psi$ defined on $\R^n\times \Gr(d,n)$ (we use \(\psi_* V\) to denote the push-forward).

\begin{defn}[{\cite[Definition 4.8]{DeRosaKolasinski}}]
    \label{def:BC}
    We say that an integrand $\Psi\in C^1(\Gr(d,n),(0,\infty))$ satisfies (BC) if for any~$T \in \Gr(d,n)$, if $\mu\neq 0$ is a positive Radon
    measure over $\Gr(d,n)$ such that
    $$\delta_{\Psi} \big[(\mathcal H^d \rstr T) \otimes \mu\big] = 0,$$
    then $\supp \mu = \{T\}$.
\end{defn}
This condition was introduced in \cite{DeRosaKolasinski}, where it was proved to be equivalent to the atomic condition \cite[Lemma 7.2]{DeRosaKolasinski}. One application of the atomic condition is given by Theorem 1.2 of \cite{DePDeRGhi}, which states that if $\Psi$ satisfies the atomic condition and $V$ is a $d$--dimensional varifold such that $\delta_\Psi[V]$ is a Radon measure, then the varifold 
$$V_* :=V\rstr \{x\in \R^n  : \theta_{d*}(x,V)>0\}\times \Gr(d,n)$$
obtained by restricting $V$ to points of positive lower $d$--density is $d$--rectifiable. Conversely, this theorem does not hold when $\Psi$ does not satisfy the atomic condition. By Proposition 3.1 of \cite{DPHRR}, the same rectifiability result holds also restricting $V$ to points of positive upper $d$--density.

Our results help to characterize the atomic condition in general codimension.
In codimension one, a $C^1$ integrand $\Psi: \Gr(n-1,n)\to (0,\infty)$ satisfies the atomic condition if and only if the function $G: v \in S^{n-1} \mapsto \Psi(v^\perp)$ is strictly convex as a function on the unit sphere $S^{n-1}$ \cite[Theorem 1.3]{DePDeRGhi}.
That is, for any positive Radon measure $\mu$ on $S^{n-1}$, if we let $v=\int w \ud \mu(w)$, then
$$|v| G\left(\frac{v}{|v|}\right) \le \int_{S^{n-1}} G(w) \ud \mu(w),$$
with equality only when $\mu$ is a point measure.

In higher codimensions, we can use Theorem~\ref{thm:construct} to show that strict polyconvexity is necessary for the atomic condition to hold. 
\begin{thm}\label{thm:main-atomic2}
    If $\Psi\in C^1(\Gr(d,n),(0,\infty))$ satisfies the atomic condition, then the even extension $\tilde \Psi$ of $\Psi$ to $\OGr(d,n)$ is strictly polyconvex.
\end{thm}

We prove Theorem~\ref{thm:main-atomic2} by showing that a varifold constructed as the limit of a minimizing sequence of rectifiable currents is stationary (Proposition~\ref{stationarylemma}), then constructing rectifiable currents that converge to varifolds of the form $(\cH^d\rstr T)\otimes \mu$. We give two such constructions, one based on ideas from \cite[Theorem 8.8]{DeRosaKolasinski} and one based on Theorem~\ref{thm:construct}.

\subsection*{Outline of paper}
In Section \ref{sec1} we will provide the preliminaries and the basic terminology that we will use throughout the paper. Section \ref{sec2} is the core of the paper, where we prove Theorem \ref{thm:construct} and use it to prove Theorem~\ref{thm:poly-multi} and Theorem~\ref{thm:poly-multi-conseq}. Finally, in Section \ref{sec:atomic}, we prove Theorem \ref{thm:main-atomic2}.

\section{Preliminaries}\label{sec1}
We first recall some basic notation from measure theory. We define the total variation of a signed measure $\mu$ over $\OGr(d,n)$ as 
$$\|\mu \|_{\TV}:=|\mu(\OGr(d,n))|=\sup \sum_i |\mu(E_i)|,$$ where the supremum is taken over all (countable) Borel partitions $\{E_i\}$ of $\OGr(d,n)$. 
We denote with $d_{\mathsf{W}}(\mu,\nu)$ the Wasserstein distance between two probability measures $\mu,\nu \in \mathcal{M}(\OGr(d,n))$. The Wasserstein distance $d_{\mathsf{W}}$ metrizes the weak-$*$ convergence of probability measures on compact spaces.

Let $\supp \mu$ be the support of a measure $\mu$, and for a Borel set \(E\),  let \(\mu\rstr E\) be the restriction of \(\mu\) to \(E\), so that \([\mu\rstr E](A)=\mu(E\cap A)\) for every Borel set $A$. 
For a current $T$ and a Borel set $E$, we denote with $\mass T$ the total mass of $T$ and with $\mass_T(E):=\mass (T\rstr E)$.

In order to construct the currents in Theorem~\ref{thm:construct}, we will introduce some notation for working with polyhedral chains. 

For a coefficient group $G$, let $\Cpoly_d(\R^n;G)\subset \cF_d(\R^n;G)$ be the space of polyhedral $d$--chains in $\R^n$ with coefficients in $G$, i.e., linear combinations of the form
\begin{equation} \label{eq:T-sum}
  T=\sum_{i=1}^k a_i [\Delta_i],
\end{equation}
where the $\Delta_i$'s are oriented $d$--dimensional convex polyhedra and $a_i\in G$.

A (polyhedral) triangulation $\tau$ of \(\R^n\) is a subdivision of \(\R^n\) into $n$--dimensional simplices such that each simplex is the convex hull of $n+1$ points in general position and any two simplices are disjoint or share a common face of dimension less than $n$. Furthermore, $\tau$ is a locally finite, i.e., any bounded set in \(\R^n\) intersects only finitely many simplices in $\tau$. All of our triangulations of $\R^n$ will be polyhedral. The following well-known lemma will be convenient.
\begin{lemma}\label{lem:polyhedral-to-simplicial}
  Let $A_1,\dots, A_k\in \Cpoly_d(\R^n;G)$. Then there is a triangulation $\tau$ of \(\R^n\) such that the $A_i$'s are all simplicial chains in $\tau$. That is, if the $d$--simplices of $\tau$ are $\sigma_1,\sigma_2,\dots$, then there are coefficients $a_{i,j}\in G$ such that only finitely many of the $a_{i,j}$'s are nonzero and 
  $$A_i = \sum_j a_{i,j} [\sigma_j], \qquad \mbox{for all $i$.}$$
\end{lemma}

One of our main tools is intersecting two transverse polyhedral chains to get a third. This construction is based on the techniques for slicing flat chains by affine planes in \cite{WhiteRectifiabilityOfFlat}.
We say that two convex polyhedra $\Delta$ and $\Lambda$ are transverse if for every face $\delta\subset \Delta$ and every face $\lambda\subset \Lambda$, 
$$\dim(\delta\cap \lambda) \le \dim(\delta)+\dim(\lambda)-n.$$
If $\Delta$ and $\Lambda$ are transverse and if $\delta\subset \Delta$ and $\lambda\subset \Lambda$, then either $\delta\cap \lambda=\emptyset$ or $\delta\cap \lambda$ is a convex polyhedron with dimension $\dim(\delta)+\dim(\lambda)-n$.

We say that two polyhedral chains $T$ and $T'$ are transverse if they can be written as sums $T=\sum_{i=1}^k a_i [\Delta_i]$ and $T'=\sum_{i=1}^{k'} a'_i [\Delta'_i]$ such that $\Delta_i$ and $\Delta_j'$ are transverse for all $i$ and $j$. 
Note that for any $T$ and $T'$, we can make $T$ and $T'$ transverse by translating either one by a small vector.
If $T$ and $T'$ are transverse, we define
$$T\cap T' := \sum_{i,j} a_i a'_j [\Delta_i\cap \Delta'_j].$$
In order for $[\Delta_i\cap \Delta'_j]$ to be well-defined, we must use the orientations on $\Delta_i$ and $\Delta'_j$ to orient $\Delta_i\cap \Delta'_j$. As in \cite[Section 3]{WhiteRectifiabilityOfFlat}, we orient $\Delta_i\cap \Delta'_j$ with a basis $w_1,\dots, w_c$ such that, if $u_1,\dots, u_a,w_1,\dots, w_c$ is a basis for $\Delta_i$ that agrees with the orientation of $\Delta_i$, and $v_1,\dots, v_b,w_1,\dots, w_c$ is a basis for $\Delta_j'$ that agrees with the orientation of $\Delta_j'$, then $u_1,\dots, u_a,$ $v_1,\dots, v_b,w_1,\dots, w_c$ is a positively oriented basis for $\R^n$. This choice of orientation guarantees that if $U\subset \R^n$ is an $n$--dimensional polyhedron transverse to $T$, then $T\cap [U]$ is equal to the restriction $T\rstr U$. 

With this choice of orientation, $\cap$ is bilinear, $A\cap B =(-1)^{\dim(A)\dim(B)} B\cap A$ and 
\begin{equation}\label{eq:leibniz-poly}
  \partial (T\cap T') = (-1)^{n-\dim T'}\partial T \cap T' + T\cap \partial T'.
\end{equation}
In this paper, we are primarily concerned with the case that $U\subset \R^n$ is an $n$--dimensional polyhedron and $T'=[U]$, in which case 
\begin{equation}\label{eq:top-dim-leibniz}
  \partial (T\cap T') = \partial T \cap [U] + T\cap \partial [U] = \partial T \rstr U + T\cap \partial [U].
\end{equation}

We now focus our attention on some properties of the weighted Gaussian image $\gamma_T$ of $T$. Recall that $\gamma_T\in \cM(\OGr(d,n))$ is the pushforward of the mass of $T$ under the Gauss map. We first observe the following subadditivity property.
\begin{lemma}\label{subadditivity}
  For any $A_1,A_2\in \bR^d(\R^n)$, 
  $$\gamma_{A_1+A_2}\leq \gamma_{A_1}+\gamma_{A_2}.$$
\end{lemma}
\begin{proof}
    By the definition of rectifiable currents, for $i=1,2$, the action of $A_i$ on any differential $d$--form $\tau$ can be expressed by
\begin{equation} \label{rect} 
\langle A_i, \tau \rangle = \int_{E_i} \langle V_i(x), \tau(x) \rangle \, \ud\mathcal H^{d}(x),
\end{equation}
where $E_i \subset \R^{n}$ is countably $d$-rectifiable and $V_i(x)$ is a $\mathcal H^{d}$--measurable, simple $d$--vector field which is parallel to the approximate tangent space of $E_i$ at $\mathcal H^{d}$--a.e. $x \in E_i$. In short, we can represent $A_i$ as $(E_i,V_i)$. 

Let $E=E_1\cup E_2$. If we extend $V_i$ by $0$ wherever it is not defined, then we can write $A_i = (E,V_i)$ and $A_1+A_2=(E,V_1+V_2)$. Let $U\subset \OGr(d,n)$ and let $C\subset \bigwedge^d\R^n$ be the cone consisting of nonzero simple $d$--vectors which are parallel to an oriented $d$-plane in $U$. Then 
\begin{equation}\label{eq:cone-gaussian}
  \gamma_{A_i}(U) = \int_E \one_C(V_i(x))|V_i(x)|\ud \cH^d(x).
\end{equation}

Let $V=V_1+V_2$. For $\mathcal H^{d}$--a.e. $x \in E$, if $V(x)\ne 0$, then $V_1(x)$ and $V_2(x)$ are real multiples of $V(x)$. In particular, if $V(x)\in C$, then $V_1(x)\in C$ or $V_2(x)\in C$ or both. In all three cases,
$$|V(x)| \le \one_C(V_1(x)) |V_1(x)| + \one_C(V_2(x))  |V_2(x)|.$$
Therefore,
\begin{multline*}
  \gamma_{A_1+A_2}(U) = \int_E \one_C(V(x)) |V(x)| \ud \cH^d(x) \\
  \le \int_E \one_C(V_1(x)) |V_1(x)| + \one_C(V_2(x)) |V_2(x)|\ud \cH^d(x) = \gamma_{A_1}(U) + \gamma_{A_2}(U),
\end{multline*}
as desired.
\end{proof}

Similarly, the map $A\mapsto \gamma_A$ is distance-decreasing. 
\begin{lemma}\label{lem:dist-decrease}
  For any $A_1,A_2\in \mathbf{R}^d(\R^n)$, 
  $$\|\gamma_{A_1} - \gamma_{A_2}\|_{\TV} \le \mass(A_1-A_2).$$
  In particular, taking $A_2=0$, this implies $\|\gamma_{A}\|_{\TV} \le \mass(A)$.
\end{lemma}
\begin{proof}
  As above, we represent $A_i$ by $(E_i,V_i)$, let $E=E_1\cup E_2$, let $U\subset \OGr(d,n)$ be a Borel set, and let $C\subset \bigwedge^d\R^n$ be the corresponding cone.

  For $\cH^d$--a.e. $x\in E$, $V_1(x)$ and $V_2(x)$ are contained in a $1$--dimensional subspace $L\subset \bigwedge^d\R^n$. The intersection $L\cap C$ is either empty, a ray from $0$, or $L\setminus \{0\}$. In all three cases, the map  $V \in L \mapsto \one_C(V)|V| \in \R$ is distance-nonincreasing. Therefore, for $\cH^d$--a.e. $x\in \R^n$, 
  \begin{equation}\label{eq:dist-diff}
    \big|\one_C(V_1(x))|V_1(x)| - \one_C(V_2(x))|V_2(x)|\big| \le |V_1(x)-V_2(x)|.
  \end{equation}
  If we integrate this over $\R^n$, we obtain
  \begin{multline*}
    |\gamma_{A_1}(U) - \gamma_{A_2}(U)| \stackrel{\eqref{eq:cone-gaussian}}{\le} \int_E \big|\one_C(V_1(x))|V_1(x)| - \one_C(V_2(x))|V_2(x)|\big|\ud \cH^d(x) \\\le \int_E |V_1(x)-V_2(x)| \ud \cH^d(x)= \mass (A_1-A_2),
  \end{multline*}
  as desired.
\end{proof}

\section{Constructing surfaces with prescribed tangent planes}\label{sec2}

In this section, we construct surfaces whose weighted Gaussian images approximate a chosen measure $\mu$. We start with the special case that $\mu$ is a sum of point measures supported on oriented $d$-planes with rational slopes, then generalize to arbitrary measures. We say an oriented $d$--plane $P\in \OGr(d,n)$ has \emph{rational slope} if it has a basis of vectors with rational coefficients. Equivalently, $P$ has rational slope if there is an integer matrix $M\in M_{n-d,n}(\Z)$ such that $P=\ker M$. 

We will show the following results.

\begin{prop} \label{prop:cycles}
  Let $P_1,\dots, P_k\in \OGr(d,n)$ be oriented $d$-planes with rational slope. Let $m_1,\dots, m_k>0$ be coefficients such that 
  $\sum^{k}_{i=1} m_i \omega_{P_i} = 0$. Let $\mu := \sum^{k}_{i=1} m_i \delta_{P_i}$.
  Then for any $\epsilon>0$, there is a polyhedral $d$--cycle $A_\epsilon\in \Cpoly_d(\R^n;\Q)$ such that 
  $\|\gamma_{A_\epsilon} - \mu \|_{\TV}<\epsilon$.
  As $\epsilon\to 0$, $\supp A_\epsilon$ converges to $[0,1]^n$ in the Hausdorff metric and the corresponding varifolds $\cV(A_\epsilon)$ converge weakly in the sense of varifolds to $(\cH^n\rstr [0,1]^n)\otimes p_* \mu$. 
\end{prop}

\begin{prop}\label{prop:cycles-real-chains}
  Let $P_1,\dots, P_{k}\in \OGr(d,n)$ be oriented $d$-planes with rational slopes. Let $P_0\in \OGr(d,n)$ be the coordinate plane $\R^d\subset \R^n$ with an orientation and $D=[0,1]^d\subset P_0$ be a unit $d$-cube.   
  Let  $m_1,\dots,m_{k} > 0$ be coefficients such that $\sum^{k}_{i=1} m_i \omega_{P_i} = \omega_{P_0}$. Let $\mu := \sum^{k}_{i=1} m_i \delta_{P_i}$. 
  
  Then for any $\epsilon > 0$, there exists $A_\epsilon\in \Cpoly_d(\R^n;\Q)$ such that $\partial A_\epsilon=\partial [D]$  and $\|\gamma_{A_\epsilon} - \mu\|_{\TV} < \epsilon$. Furthermore, as $\epsilon \to 0$, $\supp A_{\epsilon}$ converges to $D$ in the Hausdorff metric and the corresponding varifolds $\cV(A_\epsilon)$ converge weakly in the sense of varifolds to $V=(\cH^d\rstr D)\otimes p_*\mu$.
\end{prop}

The existence of $A_\epsilon$ such that $\partial A_\epsilon = \partial [D]$ and $\|\gamma_{A_\epsilon} - \mu\|_{\TV} < \epsilon$ is essentially due to Burago and Ivanov \cite{BurIva}, but we provide a new constructive proof which lets us control the geometry of $A_\epsilon$. Finally, we use these new methods to extend Proposition \ref{prop:cycles-real-chains} to multigraphs.

\begin{prop}\label{prop:multigraph-fillings}
  Let $P_0\in \OGr(d,n)$ be the coordinate $d$--plane and let $P_1,\dots, P_{k}\in \OGr^+(d,n)$ be positively oriented $d$-planes with respect to $P_0$ with rational slope. Let $D$ be a unit $d$-cube in $P_0$.
  Let $m_1,\dots,m_{k} > 0$ be coefficients such that $\sum^{k}_{i=1} m_i \omega_{P_i} = \omega_{P_0}$ and let $\mu = \sum^{k}_{i=1} m_i \delta_{P_i}$. 
  
  Then for any $\epsilon > 0$, there exists $B_\epsilon\in \Cpoly_d(\R^n;\Q)$ such that $\partial B_\epsilon=\partial [D]$, $\|\gamma_{B_\epsilon} - \mu\|_{\TV} < \epsilon$, and every tangent plane to $B_\epsilon$ is positively oriented with respect to $P_0$.
\end{prop}

In fact, $B_\epsilon$ is a multiple of the fundamental class of a multigraph, as showed in the following lemma.
\begin{lemma}
    Let $D$ be the unit cube in $\R^d$. If $B\in \Cpoly_d(\R^n;\Q)$ is such that $\partial B=\partial [D]$ and $\cH^d$-almost every tangent plane to $B$ is positively oriented with respect to $\R^d$, then there exists $0\neq Q\in \N$ and a $Q$--valued Lipschitz function $u\from D\to \mathcal{A}_Q(\mathbb{R}^{n-d})$ such that $B=Q^{-1}[\Lambda_u]$.
\end{lemma}
\begin{proof}
 Let $\pi\from \R^n\to \R^d$ be the orthogonal projection onto $\R^d$ and for $v\in \R^d$, let $\tau_v\from \R^n\to \R^n$ be the translation by $v$.
 Let $0\neq Q\in \N$ be such that $QB\in \Cpoly_d(\R^n;\Z)$ and let $B'=QB$.

 Then $\partial(\pi_\sharp(B'))=\pi_\sharp(\partial B') = Q \partial [D]$, and the constancy theorem for currents implies that 
$$
     \pi_\sharp(B')=Q [D].
$$
 
 Since $\cH^d$-almost every tangent plane to $B$ is positively oriented with respect to $\R^d$, we deduce that $D=\supp(\pi_\sharp(B'))=\pi(\supp(B'))$.
 In fact, if $v\in D$ and $\mathbb{R}^{n-d}$ is transverse to $(\tau_{-v})_\sharp(B')$, then  $[\mathbb{R}^{n-d}]\cap (\tau_{-v})_\sharp(B')$ is a polyhedral 0--chain with positive integer coefficients and mass $Q$, i.e., an element of $\mathcal{A}_Q(\mathbb{R}^{n-d})$. Define
 $$\tilde{u}(v) = [\mathbb{R}^{n-d}]\cap (\tau_{-v})_\sharp(B')\in \mathcal{A}_Q(\mathbb{R}^{n-d});$$
 this is defined for $\cH^d$-almost every $v\in D$.

 We show that $\tilde{u}$ is Lipschitz by slicing $B'$ into $1$--dimensional chains.
 For almost every oriented line $L$ through $D$, the intersection $B'_L := [\pi^{-1}(L)]\cap B'$ is equal to the graph of $\tilde{u}$ on $L\cap D$ and, by \eqref{eq:leibniz-poly},
 $$\partial B'_L = [\pi^{-1}(L)]\cap \partial[B'] = [\pi^{-1}(L)]\cap Q\partial [D] = Q \partial [L\cap D].$$ 
 The chain $B'_L$ is made up of oriented line segments in $\pi^{-1}(L)$; we say that a segment $S$ is positively oriented with respect to $L$ if $\pi$ sends $S$ to a segment of $L$ with positive orientation. Since the polyhedra making up $B$ are positively oriented with respect to $\R^d$, the segments making up $B'_L$ are positively oriented with respect to $L$. 
 
 It follows that there is no cycle $C\ne 0$ such that $\mass(B'_L)=\mass(C)+\mass(B'_L-C)$. Indeed, assume by contradiction such a cycle $C$ exists. Since $\cH^d$--almost every tangent plane to $B$ is positively oriented with respect to $\R^d$, then for almost every line $L$ through $D$ we have that $\cH^1$-almost every tangent line to $B'_L$ is positively oriented with respect to $[L\cap D]$.  By the Smirnov decomposition theorem \cite[Theorem A-Theorem C]{Smirnov}, the tangent lines to $C$ and to $B'_L$ coincide $\cH^1$--almost everywhere. Therefore, $\pi_\sharp(C)$ is a $1$--cycle in $L$ with nonnegative multiplicity, so $\pi_\sharp(C)=0$, and thus $C=0$.
 
 We conclude from Smirnov decomposition theorem \cite[Theorem A-Theorem C]{Smirnov} that we can decompose $B'_L$ into $Q$ $1$--dimensional polyhedral chains with boundary $[L\cap D]$, each of which consists of positively-oriented segments. Each such chain is the graph of a Lipschitz function on $L\cap D$, so $B'_L$ is the graph of a $Q$--valued Lipschitz function, and $\tilde{u}$ is Lipschitz on $L$, with constant depending only on $B$. Since $\tilde{u}$ is uniformly Lipschitz on almost every line, it can be extended to a Lipschitz function $u$ on $D$, and $B=Q^{-1}[\Lambda_{u}]$.
\end{proof}

The following lemma lets us generalize Propositions~\ref{prop:cycles}, \ref{prop:cycles-real-chains}, and \ref{prop:multigraph-fillings} to arbitrary measures. 
\begin{lemma}\label{lem:rational-approx}
  Let $\mu \in \mathcal{M}(\OGr(d,n))$. For any $\epsilon>0$, there are $P_1,\dots, P_k\in \OGr(d,n)$ and coefficients $a_1.\dots, a_k>0$ such that if $\mu' = \sum a_i \delta_{P_i}$, then $\mu(\OGr(d,n)) = \mu'(\OGr(d,n))$, $d_{\mathsf{W}}(\mu' , \mu)< \epsilon$, and $\int \omega_P \ud \mu'(P)$ is a multiple of $\int \omega_P \ud \mu(P)$.

  Furthermore, if there are $v>0$ and $P_0\subset \R^n$ such that $\int \omega_P \ud \mu(P) = v \omega_{P_0}$ and the support of $\mu$ consists of planes that are positively oriented with respect to $P_0$, then we can choose the $P_i$'s to be positively oriented with respect to $P_0$ as well.
\end{lemma}

The three parts of Theorem~\ref{thm:construct} follow directly from Lemma~\ref{lem:rational-approx} and  Propositions~\ref{prop:cycles}, \ref{prop:cycles-real-chains}, and \ref{prop:multigraph-fillings}.

In order to provide the proofs of the propositions above, throughout this section, for $0\neq N,M\in \N$, we denote $\langle N\rangle := \{0,\dots, N-1\}$, $\Gamma_N := \langle N\rangle^n$, $\partial \Gamma_N := \langle N\rangle^n\setminus \{1,\dots, N-2\}^n$, $\Gamma_{N,M} = \langle N\rangle^d \times \langle M \rangle^{n-d}$ and $\partial \Gamma_{N,M}:=\Gamma_{N,M}\setminus (\{1,\dots N-2\}^d \times \{1, \dots, M-2\}^{n-d})$. Moreover we define the scaling map $\rho_N\from v\in \R^n\to N^{-1}v\in \R^n$.

\subsection{Constructing closed surfaces}

We construct the surfaces in Propositions \ref{prop:cycles}-\ref{prop:cycles-real-chains}-\ref{prop:multigraph-fillings} out of cycles which we call \emph{tiles}, see Figure~\ref{fig:one-tile}. These are cycles supported in the unit cube or another fundamental domain $F$ of $\R^n$ for the action of $\Z^n$, such that the restriction of the surface to the interior of $F$ consists of intersections of $F$ with planes parallel to the $P_i$. (By a fundamental domain, we mean a \emph{polyhedral fundamental domain}, a polyhedron $F\subset \R^n$ such that the orbit $\Z^n + F$ is all of $\R^n$ and such that, for $x \in \Z^n$, the intersection $F\cap(x+F)$ is contained in $\partial F \cap \partial(x+F)$.)

We construct these tiles as follows. Let $\Z^n$ act on $\R^n$ by translation. We write this action multiplicatively, so that if $y \in \Z^n$ and $z \in \R^n$, then $yz$ is the translation $z+y$. Likewise, for $B\subset \R^n$, we write $yB:=y+B$.

For any plane $P$ with rational slope, choose $M\in M_{n-d,n}(\Z)$ such that $P=\ker M$.  Then the linear map $M\from \R^n\to \R^{n-d}$ descends to a map $\hat{M}\from \R^n/\Z^n\to \R^{n-d}/\Z^{n-d}$, and the set $\Sigma_P=\hat{M}^{-1}(\zero)\subset \R^n/\Z^n$ is a submanifold of $\R^n/\Z^n$ whose tangent planes are parallel to $P$. Indeed, if we let $\tilde{\Sigma}_P$ be the lift of $\Sigma_P$ to $\R^n$, then $\tilde{\Sigma}_P$ is a union of planes parallel to $P$. We orient $\Sigma_P$ and $\tilde{\Sigma}_P$ parallel to $P$, so that $\gamma_{\Sigma_P}$ is a point measure supported at $P$. Let
\begin{equation}\label{eq:T-normalization}
    T_P=\frac{[v_P + \Sigma_P]}{\mass\Sigma_P} \in \Cpoly_d(\R^n/\Z^n;\R),
\end{equation}
so that $\gamma_{T_P}=\delta_P$, where the $v_P$ are small vectors chosen so that there is no cancellation when adding $T_P$ and $T_{-P}$. Let $\tilde{T}_P$ be the lift of $T_P$ to $\R^n$, so that $\tilde{T}_P$ is a sum of parallel planes.

\begin{figure}
  \begin{centering}
    \includegraphics{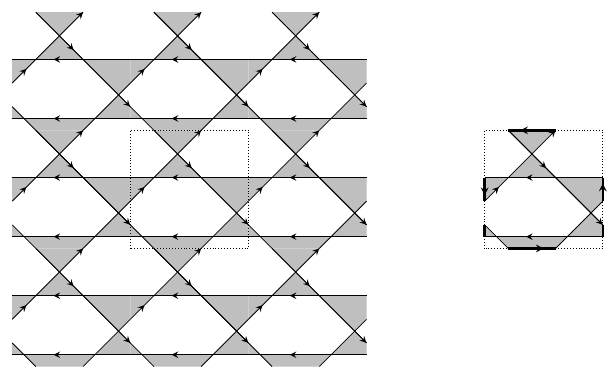}
  \end{centering}
  \caption{\label{fig:one-tile}(left) An example with $d=1$, $n=2$. The polyhedral chain $\tilde{S}$ is a sum of sets of parallel $d$--planes such that the corresponding $d$--vectors (in this case, $m_i \omega_{P_i} = (1,1), (1,-1)$, and $(-2,0)$) sum to zero. This condition implies that $\tilde{S}=\partial \tilde{Q}$ for some $(d+1)$--chain $\tilde{Q}$ (in gray). The dotted square is a  fundamental domain $F$ for the action of $\Z^n$. 
  \endgraf
  \medskip
  (right) A tile $S_F = \partial(\tilde{Q}\rstr F)$. $S_F$ is a $1$--cycle that can be written as the sum of $\tilde{S}\rstr F$. It consists of $d$--planes parallel to the $P_i$'s (thin lines), and $\tilde{Q}\cap \partial [F]$, which is supported on $\partial F$ (thick lines).}
\end{figure}

Let $P_1,\dots, P_k\in \OGr(d,n)$ be oriented $d$-planes with rational slope. Let $m_1,\dots, m_k>0$ be coefficients such that 
  $\sum^{k}_{i=1} m_i \omega_{P_i} = 0$. Let $\mu = \sum_{i=1}^k m_i \delta_{P_i}$. Then $S=\sum_{i=1}^k m_i T_{P_i}\in \Cpoly_d(\R^n/\Z^n;\R)$ is a $d$--cycle with $\gamma_S = \mu$. Since the barycenter of $\gamma_S$ on $\bigwedge^d\R^n$ is $0$, then $S$ is null-homologous in $\R^n/\Z^n$, and there is a polyhedral $(d+1)$--chain $Q\in \Cpoly_{d+1}(\R^n/\Z^n;\R)$ such that $\partial Q = S$. Let $\tilde{Q}$ be the lift of $Q$ to $\R^n$ and let $\tilde{S}$ be the lift of $S$. Then $\tilde{Q}$ is a locally finite polyhedral $(d+1)$--chain and $\tilde{S} = \partial \tilde{Q}$ is a sum of $d$--planes parallel to the $P_i$'s.

For any $n$--dimensional polyhedron $F\subset \R^n$ which is transverse to $\tilde{Q}$, let $S_F =\partial  (\tilde{Q}\cap [F])$. Then, by \eqref{eq:top-dim-leibniz},
\begin{equation}\label{eq:intersection-formula}
  S_F = \partial  (\tilde{Q}\cap [F]) = \partial \tilde{Q}\rstr F + \tilde{Q}\cap \partial F = \tilde{S}\rstr F + \tilde{Q}\cap \partial [F].
\end{equation}

Then $S_F$ consists of a chain $\tilde{S}\rstr F$ made up of planes parallel to the $P_i$'s and a chain $\tilde{Q}\cap \partial [F]$ with support lying in $\partial F$. When $F$ is large, we expect $\tilde{Q}\cap \partial [F]$ to have small mass compared to $\tilde{S}\rstr F$. Hence, since $\gamma_{\tilde{S}\smrstr F}$ is a multiple of $\mu$, we expect $\gamma_{S_F}$ to be close to a multiple of $\mu$. We prove the propositions by taking the union of translates of a fundamental domain $F$, so that we can control both of these terms, see Figure~\ref{fig:tile-grid}(top).

\begin{figure}
  \begin{center}
    \includegraphics{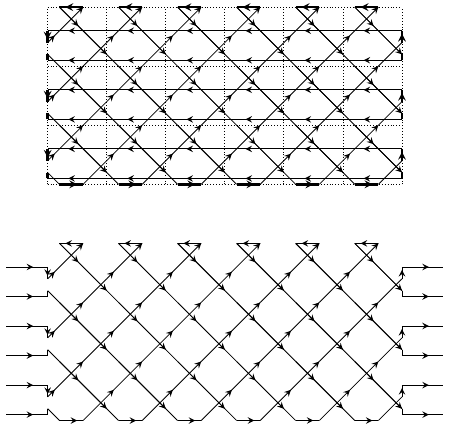}
  \end{center}
  \caption{\label{fig:tile-grid}(top) We construct closed surfaces with prescribed tangents by arranging the tiles in a grid. The boundary components of adjacent tiles cancel out, leaving the interior parts, which are parallel to $P_0=\langle (-2,0)\rangle$, $P_1=\langle (1,1)\rangle$, and $P_2=\langle (1,-1)\rangle$. When the grid is large, the interior has much larger mass than the boundary.
  \endgraf
  \medskip
  (bottom) We produce a surface with boundary on $P_0$ and Gaussian image concentrated on $P_1$ and $P_2$ by adding lines parallel to $-P_0$. This cancels out the faces that are parallel to $P_0$. We then add annuli (not shown) so that the boundary of the surface lies on $P_0$. When the grid has large height and even larger width, all but a small fraction of the surface is parallel to $P_1$ or $P_2$.
  }
\end{figure}

First, we prove Proposition~\ref{prop:cycles}.
\begin{proof}[Proof of Proposition~\ref{prop:cycles}]  
  Let $\tilde{S}$ and $\tilde{Q}$ be as above, so that $\tilde{S}$ is a sum of $d$--planes and $\partial \tilde{Q}=\tilde{S}$. 
  Let $F=[0,1]^n \subset \R^n$ be the unit cube; after potentially translating $\tilde{S}$ and $\tilde{Q}$, we may suppose that $F$ is transverse to $\tilde{Q}$. 
  Let 
  $$S_F=\partial(\tilde{Q}\rstr F) = \tilde{S}\rstr F + \tilde{Q}\cap \partial[F]$$
  as in  \eqref{eq:intersection-formula}. Then $\partial S_F = \partial^2(\tilde{Q}\rstr F)=0$, so $S_F$ is a cycle. 

  For $v\in \Z^n$, let $\tau_v\from \R^n\to \R^n$ be the translation by $v$. We abbreviate the translate $(\tau_v)_\sharp S_F$ by $vS_F = S_{vF}$. For $0\neq N \in \N$, by definition of $\Gamma_N$, the translates $v F, v\in \Gamma_N$ tile $[0,N]^n$ by unit cubes, and we let 
  \begin{equation}\label{eq:def-C}
    C_N := \sum_{v\in \Gamma_N} v S_F.
  \end{equation}
  and $A_N = N^{-(n-d)} (\rho_N)_\sharp(C_N)$. These are cycles with $\supp C_N\subset [0,N]^n$, $\supp A_N \subset F$ and $\mass A_N=\mass S_F$. We claim that $\gamma_{A_N}$ converges to $\mu$ and that the varifolds $\cV(A_N)$ corresponding to $A_N$ converge to $(\cH^n\rstr F)\otimes p_*\mu$.

  For $0\neq N\in \N$, let $E_N := \Gamma_NF = [0,N]^n$. By linearity, 
  \begin{equation}\label{eq:cn-linear}
    C_N = \sum_{v\in \Gamma_N} vS_{F} = \partial \left(\tilde{Q} \cap \sum_{v\in \Gamma_N} [vF]\right) = \partial\left(\tilde{Q} \cap \left[E_N\right]\right) = \tilde{S}\rstr E_N + \tilde{Q}\cap \partial[E_N].
  \end{equation}
  Let $U_N=\tilde{S}\rstr E_N$ and $T_N=\tilde{Q}\cap \partial[E_N]$; we will bound the weighted Gaussian images of $U_N$ and $T_N$. 

  Since $[E_N] = \sum_{v\in \Gamma_N}[vF]$, $U_N$ consists of $N^n$ translates of $\tilde{S}\rstr F$, and 
  $$\gamma_{U_N} = N^n \gamma_{\tilde{S}\smrstr F} = N^n\gamma_S = N^n \mu.$$
  Likewise,
  $$T_N = \sum_{v\in \Gamma_N} \tilde{Q} \cap \partial [v F].$$
  We have $\supp(\tilde{Q}\cap \partial [v F])\subset vF$, however $\supp T_N\subset \partial E_N$, so when $v\in \{1,\dots, N-2\}^n$ then $\tilde{Q}\cap \partial [v F]$ does not contribute to the mass of $T_N$. That is,
  $$\mass T_N = \mass_{T_N}(\partial E_N) \le \sum_{v\in \Gamma_N} \mass_{\tilde{Q} \cap \partial [v F]}(\partial E_N) \le \sum_{v\in \partial \Gamma_N} \mass(\tilde{Q}\cap \partial [v F]).$$
  Therefore, 
  \begin{equation}\label{eq:tn-mass-bound}
    \mass T_N \le \#\partial \Gamma_N \cdot \mass(\tilde{Q}\cap \partial [F]) \le c N^{n-1},
  \end{equation}
  where $c=2n\mass(\tilde{Q}\cap \partial [F])$.

  By scaling, $\gamma_{A_N} = N^{-n} \gamma_{T_N+U_N}$, so by the above bounds and Lemma~\ref{lem:dist-decrease},
  $$\|\gamma_{A_N} - \mu\|_{\TV} = N^{-n} \|\gamma_{T_N+U_N} - \gamma_{U_N}\|_{\TV} \le N^{-n} \mass T_N \le cN^{-1}$$
  and thus $\gamma_{A_N}\to \mu$ in the total variation norm as $N\to \infty$.

  We now prove that $\supp A_N$ converges to $F$ in the Hausdorff metric. Since $\supp A_N\subset F$, it is enough to observe that, for every $x\in F$ there exists $v\in \Gamma_N$ such that $x\in \rho_N(vF)$. Since $\supp A_N\cap \rho_N(vF)\neq \emptyset$ and the diameter of $\rho_N(vF)$ converges to $0$ as $N\to \infty$, we deduce that there exists $x_N\in \supp A_N$ such that $x_N \to x$ as $N\to \infty$.

  We are now left to prove the varifold convergence. By \eqref{eq:tn-mass-bound}, it is straightforward that $\lim_{N\to \infty} \mass( N^{-(n-d)} (\rho_N)_\sharp(T_N))=0$, hence we compute the following limit $W$ in the sense of varifolds:
  $$W:=\lim_{N\to \infty}\cV(A_N) = \lim_{N\to \infty}\cV\left(N^{-(n-d)} (\rho_N)_\sharp(T_N + U_N)\right)=\lim_{N\to \infty}\cV\left(N^{-(n-d)} (\rho_N)_\sharp(U_N)\right).$$
  Let $\tilde{S}_N = ((\rho_N)_\sharp\tilde{S})$. We can write $(\rho_N)_\sharp(U_N) = \tilde{S}_N \rstr F$, so 
  $$W=\lim_{N\to \infty}\cV\left(N^{-(n-d)} \tilde{S}_N\rstr F\right).$$
  
  By construction, for any $x\in \Z^n$, we have $\gamma_{\tilde{S}\smrstr xF} = \mu$ and thus $\gamma_{\tilde{S}_N \smrstr \rho_N(xF)} = N^{-d}\mu$. Therefore, for any compactly supported smooth function $f\in C^\infty(\R^n\times \Gr(d,n))$,
  \begin{align*}
    \cV(\tilde{S}_N\rstr F)[f] & = \sum_{x\in \{0,\frac{1}{N},\dots,\frac{N-1}{N}\}^n} \sum_{i=1}^k N^{-d} m_i \left(f(x,p(P_i)) + O(N^{-1})\right) \\
    & = N^{n-d} \int_{[0,1]^n\times \Gr(d,n)} f(x,p(P)) \ud x \ud \mu(P) + O(N^{n-d-1}),
  \end{align*}
  where the implicit constant in the error term may depend on $f$.
  Thus 
  $$\cV(W)[f] = \lim_{N\to \infty} N^{-(n-d)} \cV(\tilde{S}_N\rstr F)[f] = \int_{[0,1]^n\times \Gr(d,n)} f(x,p(P)) \ud x \ud \mu(P),$$ 
  and $\cV(A_N)$ converges to $W = (\cH^n\rstr F)\otimes p_*\mu$ in the sense of varifolds, as desired.
\end{proof}

\subsection{Constructing chains with planar boundaries}

We prove Proposition~\ref{prop:cycles-real-chains} by modifying the previous construction to produce a chain whose boundary lies on $P_0$, see Figure~\ref{fig:tile-grid}(bottom).
\begin{proof}[Proof of Proposition~\ref{prop:cycles-real-chains}]
  We proceed similarly to the proof of Proposition~\ref{prop:cycles}. Let $\tilde{\mu} = \mu + \delta_{-P_0}$ so that $\tilde{\mu}$ has barycenter equal to $0$ on $\bigwedge^d\R^n$. Let $S = \sum^k_{i=1}m_iT_{P_i} + T_{-P_0}\in \Cpoly_d(\R^n/\Z^n; \R)$. If necessary, we translate the $T_{P_i}$'s so that there is no cancellation in this sum and thus $\gamma_S = \tilde{\mu}$. 
  
  As in the proof of Proposition~\ref{prop:cycles}, there is a $Q\in \Cpoly_{d+1}(\R^n/\Z^n; \R)$ such that $\partial Q = S$, and we let $\tilde{S}$ and $\tilde{Q}$ be the lifts of $S$ and $Q$ respectively, so that $\partial \tilde{Q} = \tilde{S}$. 
  Let $F=[0,1]^n\subset \R^n$ be the unit cube; after potentially translating $\tilde{S}$ and $\tilde{Q}$, we may suppose that $F$ is transverse to $\tilde{Q}$. Let $S_F = \partial (\tilde{Q}\rstr F)$.
  
  Let $0\neq N,M\in \N$ and let $E=E_{N,M}=\Gamma_{N,M} F = [0,N]^d\times [0,M]^{n-d}$. We define a cycle
  \begin{multline*}
    C_{N,M} = \sum_{v\in \Gamma_{N,M}} vS_F = \partial  (\tilde{Q}\rstr{E}) = \tilde{S}\rstr E + \tilde{Q}\cap \partial [E]
    = \tilde{Q}\cap \partial [E] + \sum_i m_i \tilde{T}_{P_i}\rstr E + \tilde{T}_{-P_0}\rstr E.
  \end{multline*}
  
  Let $B_{N,M} = C_{N,M} - (\tilde{T}_{-P_0}\rstr E)$. Since $\tilde{T}_{-P_0}$ is a sum of parallel planes, there is a set $V\subset [0,M]^{n-d}$ (a translate of $\langle M \rangle^{n-d}\subset \R^{n-d}\subset \R^n$) such that 
  $$\tilde{T}_{-P_0}\rstr E = -\sum_{v\in V} v [P_0]\rstr E = -\sum_{v\in V} v [[0,N]^d].$$

  We will see that $\gamma_{B_{N,M}}$ is close to a multiple of $\mu$ when $M$ and $N$ are large. We will construct $A_{N,M}$ by adding some annuli to $B_{N,M}$ and rescaling.

  Since $C_{N,M}$ is a cycle,
  \begin{equation}\label{eq:partial-b0}
    \partial B_{N,M} = \partial C_{N,M} + \sum_{v\in V} v \partial [[0,N]^d] = \sum_{v\in V} v \partial [[0,N]^d].
  \end{equation}
  For each $v\in V$, let $R_v$ be an annulus with $\supp R_v\subset \partial E$ such that $\partial R_{v} = \partial [[0,N]^d] - v\partial [[0,N]^d]$ and $\mass (R_{v}) \lesssim |v| N^{d-1}$. Let
  \begin{equation}\label{eq:def-ANM}
    A_{N,M} = M^{-(n-d)} (\rho_N)_\sharp(B_{N,M} + \sum_{v\in V} R_v).
  \end{equation}
  Then
  $$\partial A_{N,M} = M^{-(n-d)} (\rho_N)_\sharp(\#V \partial [[0,N]^d]) = (\rho_N)_\sharp(\partial [[0,N]^d]) = \partial [[0,1]^d]$$
  is the boundary of the unit cube in $P_0$.

  Now we calculate $\gamma_{A_{N,M}}$. We have
  $$\gamma_{A_{N,M}} = N^{-d}M^{-n+d} \gamma_{B_{N,M}+ \sum_{v\in V} R_v},$$
  so by Lemma~\ref{lem:dist-decrease},
  $$\|\gamma_{A_{N,M}}-\mu\|_{\TV} \le N^{-d}M^{-n+d}\mass\bigg(\sum_{v\in V} R_v\bigg) + \|N^{-d}M^{-n+d}\gamma_{B_{N,M}}-\mu\|_{\TV}.$$
  We bound each of these terms separately.
  
  First, we can write
  \begin{equation}\label{eq:alt-def-BNM}
    B_{N,M} = \tilde{Q}\cap \partial [E_{N,M}] + \sum_i m_i \tilde{T}_{P_i}\rstr E_{N,M}.
  \end{equation}
  None of the terms in this sum cancel, so 
  $$\gamma_{B_{N,M}} = \gamma_{\tilde{Q}\cap \partial [E]} + N^d M^{n-d} \mu.$$
  As in the proof of Proposition~\ref{prop:cycles}, we compute
  \begin{equation}\label{eq:mass-bound0}
  \begin{split}
      \mass(\tilde{Q}\cap \partial [E]) &\le \#\partial \Gamma_{N,M}\cdot  \mass(\tilde{Q}\cap \partial [F])\\
      &\leq (N^d M^{n-d} - (N-2)^d(M-2)^{n-d})\mass(\tilde{Q}\cap \partial [F]),
      \end{split}
  \end{equation}
  so by Lemma~\ref{lem:dist-decrease},
  \begin{equation}\label{eq:boundary-bound}
    \big\|\frac{\gamma_{B_{N,M}}}{N^dM^{n-d}} - \mu\big\|_{\TV} \le \frac{N^d M^{n-d} - (N-2)^d(M-2)^{n-d}}{N^dM^{n-d}}\mass(\tilde{Q}\cap \partial [F]) \lesssim \frac{1}{N} + \frac{1}{M}.
  \end{equation}

  Second, since $|v|\lesssim M$ for all $v\in V$,
  \begin{equation}\label{eq:annulus-bound}
    \frac{\mass(\sum_{v\in V} R_v)}{N^dM^{n-d}} \lesssim \frac{\#V\cdot M N^{d-1}}{N^dM^{n-d}} = \frac{M^{n-d+1} N^{d-1}}{N^dM^{n-d}} = \frac{M}{N}.
  \end{equation}
  
  Combining these bounds, we get
  $$\|\gamma_{A_{N,M}} - \mu\|_{\TV} \lesssim \frac{M}{N} + \frac{1}{M} + \frac{1}{N}.$$
  In particular, choosing $N=M^2$, we have $\|\gamma_{A_{M^2,M}} - \mu\|_{\TV} \to 0$ as $M\to \infty$.

 Proving that $\supp A_{M^2,M}$ converges to $[0,1]^d$ in the Hausdorff metric is done with the very same argument used for the Hausdorff convergence in Proposition~\ref{prop:cycles}.

  Finally, we consider the corresponding varifolds $\cV(A_{M^2,M})$. By \eqref{eq:def-ANM} and \eqref{eq:alt-def-BNM},
  $$A_{M^2,M} = M^{-(n-d)} (\rho_{M^2})_\sharp\left[\tilde{Q}\cap \partial [E_{M^2,M}] + \tilde{S}' \rstr E_{M^2,M} + \sum_{v\in V} R_v\right],$$
  where $\tilde{S}'=\sum_i m_i \tilde{T}_{P_i}$. By \eqref{eq:mass-bound0} and \eqref{eq:annulus-bound},
  \begin{align*}
    \lim_{M\to \infty} \cV(A_{M^2,M}) & = \lim_{M\to \infty} \cV\left(M^{-(n-d)} (\rho_{M^2})_\sharp\left[\tilde{S}' \rstr E_{M^2,M}\right]\right).
  \end{align*}
  Let $T_M = M^{-(n-d)} (\rho_{M^2})_\sharp\left[\tilde{S}' \rstr E_{M^2,M}\right]$. Then 
  $$\gamma_{T_M}=M^{-(n-d)}M^{-2d} M^{2d+(n-d)}\mu = \mu$$
  for all $M$ and $\supp T_M\subset \rho_{M^2}(E_{M^2,M})$.  Arguing as in the proof of the varifold convergence of Proposition~\ref{prop:cycles}, we obtain that, as $M\to \infty$, $\cV(T_M)$ converges as a varifold to $(\cH^d\rstr [0,1]^d)\otimes p_* \mu$, as desired.
\end{proof}

\subsection{Constructing multigraphs}

\begin{figure}
  \begin{center}
    \includegraphics{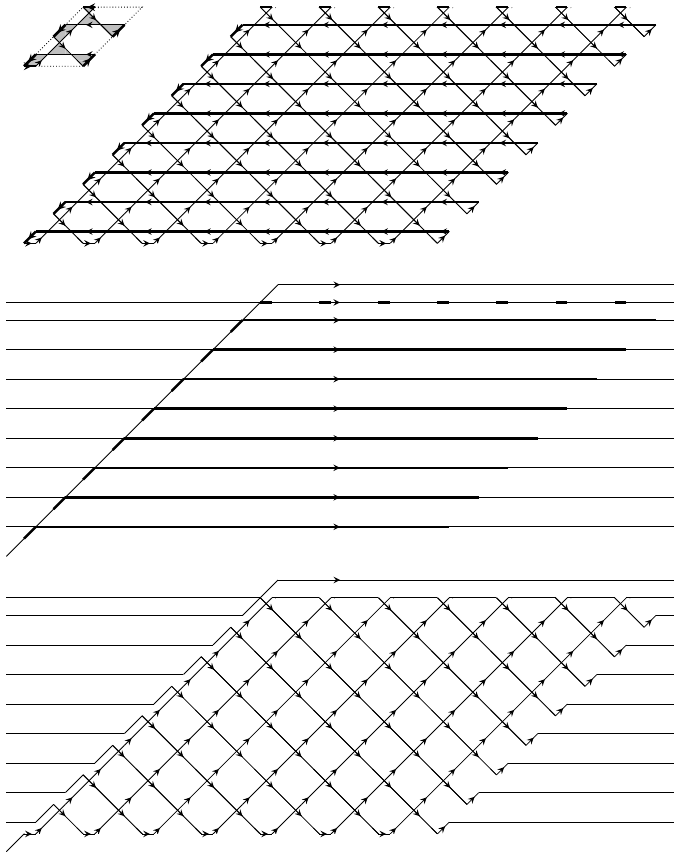}
  \end{center}
   \caption{\label{fig:slanted-grid} We construct a multigraph with prescribed tangents by changing the construction in Figure~\ref{fig:tile-grid} so that $F$ is a parallelogram. This makes every face of $S_F$ either positively or negatively oriented. 
  \endgraf
  \medskip
    To make this a multigraph, we add positively-oriented graphs that cancel out the negatively-oriented faces. These come in two types: graphs parallel to $-P_0$ and graphs containing the negatively-oriented faces on the boundary of the tiling.
  \endgraf
  \medskip
    All of the faces of the resulting surface are positively oriented, so it is a multigraph. We can add annuli (not shown) so that the boundary lies on $P_0$. When the grid has large height and even larger width, most of the lines in the multigraph are parallel to $P_1$ and $P_2$.}
\end{figure}

In this section, we prove Proposition~\ref{prop:multigraph-fillings}. 
The main idea of the proof is to construct a filling as in Proposition~\ref{prop:cycles-real-chains} such that the weighted Gaussian image of the filling is close to $\mu$. Since $\mu$ consists of positively oriented planes, the filling is positively oriented except on a small chain, and we add some additional graphs to the filling to cancel out the negatively oriented part. See Figure~\ref{fig:slanted-grid} for a schematic example, though note that some of the technical details of the construction will differ from the figure.

The negatively oriented parts of these fillings consist of translates of negatively oriented simplices. The following lemma will help us cover these simplices by graphs.
We decompose $\R^n = \R^d\times \R^{n-d}$ and let $\pi\from \R^n\to \R^d$ and $\pi^\perp\from \R^n \to \R^{n-d}$ be the orthogonal projections. 
\begin{lemma}\label{lem:extend-graphs}
  Let $Y\subset \Z^n$ be a subset which can be written as the graph of a Lipschitz function $f\from \pi(Y)\to \R^{n-d}$.
  Let $\delta\subset \R^n$ be a $d$--simplex such that $\pi$ is injective on $\delta$ and $\diam \delta \le \frac{1}{2}$. Then there is a Lipschitz function $g\from \R^d\to \R^{n-d}$ with $\Lip(g)$ depending on $\delta$ and $\Lip(f)$ such that $\bigcup_{y\in Y} y\delta \subset \Lambda_{g}$.
  
  Furthermore, we may choose $g$ so that 
  \begin{equation}\label{eq:extend-bounds}
    g(\R^d) \subset \conv(\pi^\perp(\bigcup_{y\in Y} y\delta)),
  \end{equation}
  where $\conv(X)$ is the convex hull of $X$.
\end{lemma}
\begin{proof}
  Since $Y$ is contained in a graph, $\pi$ is injective on $Y$, and the projections $\pi(y\delta)$ are all disjoint. Indeed, since $\diam \delta \le \frac{1}{2}$, the distance between two such projections is at least $\frac{1}{2}$. Let $U = \bigcup_{y\in Y} y \delta$. Then $\pi$ is injective on $U$, so there is a unique function $\tilde g\from \pi(U) \to \R^{n-d}$ such that $\Lambda_{\tilde g}=U$. This function $\tilde g$ is Lipschitz with constant depending only on $\delta$ and $\Lip(f)$, so by Kirszbraun's theorem, we can extend $\tilde g$ to a function $g \from \R^d\to \R^{n-d}$ with the same Lipschitz constant such that $\bigcup_{y\in Y} y\delta \subset \Lambda_{g}$.  Let $q\from \R^{n-d} \to \conv(\pi^\perp(\bigcup_{y\in Y} y\delta))$ be the closest-point projection; this projection is distance-decreasing, so $q\circ g$ satisfies the desired properties.
\end{proof}

Now we prove the proposition.

  Let $P_0$ be the coordinate $d$--plane $\R^d\subset \R^n$, and let $P_1,\dots, P_k$ be positively oriented planes with rational slope. Let $m_1,\dots, m_k>0$ be coefficients such that $\sum^{k}_{i=1} m_i \omega_{P_i} = \omega_{P_0}$. Let $\mu = \sum^{k}_{i=1} m_i \delta_{P_i}$.

  Since many of the constants in this proof depend on $d$, $n$, and the $P_i$'s, we will write $f \lesssim g$ to denote $f\lesssim_{d,n,P_1,\dots, P_k}g$.
  
  As in Proposition~\ref{prop:cycles-real-chains}, let 
  $$S = -T_{P_0} + \sum_{i=1}^k m_i T_{P_i}\in \Cpoly_d(\R^n/\Z^n;\R)$$
  and let $\tilde{S}$ be the lift of $S$ to $\R^n$, so that $\tilde{S}$ is a sum of planes parallel to the $P_i$'s. We can take $\tilde{T}_{P_0} = \Z^{n-d} P_0$.
  
  We say that two planes $V,W\subset \R^n$ are transverse if $\dim(V\cap W) = \dim(V) + \dim(W) - n$. Note that if $\dim(V) + \dim(W) \ge n$ and $V$ and $W$ are transverse, then so are any two planes parallel to $V$ and $W$. If $V$ and $W$ are transverse and $V'$ and $W'$ are generic planes containing $V$ and $W$ or generic subspaces of $V$ and $W$, then $V'$ and $W'$ are transverse. 

  As before, $S$ is null-homologous in $\R^n/\Z^n$, so there is a $(d+1)$--chain $Q_0$ such that $\partial Q_0 = S$. By Lemma~\ref{lem:polyhedral-to-simplicial}, there is a triangulation $\tau$ of $\R^n/\Z^n$ such that $Q_0$ is a simplicial chain in $\tau$ and thus $\supp S=\supp \partial Q_0$ is a subcomplex of $\tau$. Let $\tau'$ be the barycentric subdivision of $\tau$. Perturbing the vertices of $(\tau')^{(0)}\setminus \supp S$ fixes $\partial Q_0$ and perturbs the tangent planes of $Q_0$, so we can assume that there is a generic perturbation $Q$ of $Q_0$ such that the tangent planes of $Q$ are all transverse to $P_0^\perp$. Let $\tilde{Q}$ be the lift of $Q$ to $\R^n$.

  We proceed similarly to the proof of Proposition~\ref{prop:cycles-real-chains}, with modifications to the choice of the fundamental domain $F$ and to the set $\Gamma\subset \Z^n$.
  Let $F\subset (-1,1)^n$ be a polyhedral fundamental domain for the action of $\Z^n$. After a perturbation, we may suppose that $F$ is transverse to $\tilde{Q}$ and the intersection of any tangent plane of $\partial F$ and any tangent plane of $\tilde{Q}$ is transverse to $P_0^\perp$.
  
  For $r>0$, let $I_r := [-r,r]$. We view $\R^n$ as $\R^d\times \R^{n-d}$, so that $I_r^d\times I_s^{n-d}\subset \R^n$. For $r\in \mathbb{N}$, let $J_r:=\{-r,\dots,r\}\subset \Z$. Let $I_{r,s} := I_r^d\times I_s^{n-d}$ and $J_{r,s} := J_r^d\times J_s^{n-d}\subset \Z^n$.

  Let $K\from \R^n\to \R^n$ be the linear transformation 
  $$K(x_1,\dots, x_n) = (x_1 + x_n,\dots, x_d+x_n, x_{d+1},\dots, x_n).$$ For $M,N>0$, let $\Xi_{N,M} := K(J_{N,M})\subset \Z^n$. Note that since $F$ is a closed subset of $(-1,1)^n$,
  $$K^{-1}(F)\subset (-2,2)^{d}\times (-1,1)^{n-d},$$
  so
  $$\Xi_{N,M} F = K(K^{-1}(\Xi_{N,M} F)) = K(J_{N,M} K^{-1}(F)) \subset \inter(K (I_{N+2,M+1})).$$

  Conversely, if $v\in \Z^n$ and $v K^{-1}(F)$ intersects $I_{N-2,M-1}$, then $v\in J_{N,M}$. Since $K^{-1} (F)$ is a fundamental domain,
  $$I_{N-2,M-1} \subset  J_{N,M}K^{-1}(F)$$
  and thus
  $$K (I_{N-2, M-1}) \subset \Xi_{N,M} F.$$
  Putting these together,
  \begin{equation}\label{eq:E-bounds}
    K (I_{N-3, M-3}) \Subset \Xi_{N,M} F \Subset K (I_{N+3,M+3})
  \end{equation}
  for all $M,N \in \N$, where $S\Subset T$ denotes that $\clos(S)$ is compact and $\clos(S)\subset \inter(T)$.

  Let $M,N \in \N$ such that $M\ge \max\{\epsilon^{-1},50\}$ and $N\ge \max\{\epsilon^{-1}M,50\}$. Let $\Xi := \Xi_{N,M}$ and let $E := \Xi F$.   
  Let $A_0 := \partial(\tilde{Q}\rstr E).$
  Then, by \eqref{eq:intersection-formula} and our choice of $\tilde{S}$
  \begin{equation}\label{eq:a0-break}
    A_0 = \tilde{S}\cap [E] + \tilde{Q}\cap \partial [E]  = \left( \sum_{i=1}^k m_i \tilde{T}_{P_i} \rstr E\right) - \tilde{T}_{P_0} \rstr E + \tilde{Q}\cap \partial [E].
  \end{equation}
  We will see that $\gamma_{A_0}$ is close to a multiple of $\mu$, but $A_0$ is not a multigraph because $- \tilde{T}_{P_0}$ is negatively oriented and $\tilde{Q}\cap \partial [E]$ contains negatively oriented simplices. We will cancel the negatively oriented simplices by covering $\supp(\tilde{T}_{P_0} \cap E)$ and $\supp(\tilde{Q}\cap \partial [E])$ by positively oriented Lipschitz graphs.

  We start with $\supp(\tilde{T}_{P_0}\rstr E) = \supp(\tilde{T}_{P_0})\cap E$. Since $\supp(\tilde{T}_{P_0}) = \Z^{n-d} P_0$,  \eqref{eq:E-bounds} implies that 
  $$\supp(\tilde{T}_{P_0}\rstr E) \subset (\{0\}^d\times J_{M+3}^{n-d}) P_0 \cap I_{N+2M,M+3}.$$
  For each $v\in J^{n-d}_{M+3}$, let $f_v\from I_{N+3M}^{d}\to I_{2M}^{n-d}$ be a $2$--Lipschitz function such that $f_v(x) = v$ for $x\in I_{N+2M}^{d}$ and $f_v(x) = 0$ for $x\in \partial I_{N+3M}^{d}$. Then the graph $\Lambda_{f_v}$ contains $v P_0\cap E$ and $\partial \Lambda_{f_v}=\partial I_{N+3M}^{d}$ lies on $P_0$.

  Next, we cover $\supp \tilde{Q}\cap \partial [E]$ by positively oriented Lipschitz graphs. We first describe $\tilde{Q}\cap \partial [E]$. 
  Let $W := \tilde{Q}\cap \partial [F]$ so that
  $$\tilde{Q}\cap \partial [E] = \sum_{v\in \Xi}\tilde{Q}\cap \partial [vF] = \sum_{v\in \Xi} v W.$$
  By \eqref{eq:E-bounds}, for $R=6$, we have
  $$\Xi_{N-R,M-R}F\Subset \partial E \Subset \Xi_{N+R,M+R}F.$$
  Consequently, if $(vF)\cap \partial E \ne \emptyset$, then $v\in \Xi_{N+R,M+R}\setminus \Xi_{N-R,M-R}$. Hence
  $$\supp(\tilde{Q}\cap \partial[E]) \subset \bigcup_{x\in X_R} \supp(x W),$$
  where $X_R := \Xi_{N+R,M+R}\setminus \Xi_{N-R,M-R}$. Note that $\#X_R\approx N^{d}M^{n-d-1}$.
  
  This set can be covered by Lipschitz graphs, as proved below:
  \begin{lemma}\label{lem:cover-XR}
    There is an $L\lesssim M^{n-d-1}$ and a collection of Lipschitz functions $g_1,\dots, g_L\from \R^d\to [-2M,2M]^{n-d}$ with $\Lip(g_i)\lesssim 1$ such that $\supp g_i\subset I_{N+3M}^d$ for all $i$ and $X_R\subset \bigcup_{i=1}^L\Lambda_{g_i}$. 
  \end{lemma}
  \begin{proof}
    We write $X_R = K(J_{N+R,M+R}\setminus J_{N-R,M-R})$. For every $v=(v_1,\dots, v_n) \in J_{N+R,M+R}\setminus J_{N-R,M-R}$, either $|v_i|\in (N-R,N+R]$ for some $i\le d$ or $|v_i|\in (M-R,M+R]$ for some $i>d$.
    For $i=1,\dots, d$, let 
    $$Y_i := \{(v_1,\dots, v_n) \in J_{N+R,M+R}\mid |v_i| \in (N-R,N+R]\},$$
    and let 
    $$Z := J_{M+R}^{n-d}\setminus J_{M-R}^{n-d}.$$
    Then 
    $$X_R = K(J^d_{N+R}\times Z) \cup \bigcup_{i=1}^d K(Y_i), \quad \mbox{ and } \quad K(J^d_{N+R}\times Z) \subset \bigcup_{z\in Z} zP_0.$$

    Since $X_R\subset I_{N+M+2R,M+R}\subset I_{N+2M,2M}$, for every $z\in Z$, we can define a Lipschitz function $k_z$ such that $k_z(x)=z$ for $x\in I_{N+2M}^d$ and $\supp k_z\subset I_{N+3M}^d$. There are $\#Z\approx M^{n-d-1}$ such functions and their graphs cover $K(J^d_{N+R}\times Z)$. We can take these functions to have Lipschitz constant at most $2\sqrt{n}$.

    Likewise, we can cover $K(Y_i)$. Let $q\from \R^{n-d}\to [-2M,2M]^{n-d}$ be the closest-point projection, and for $i=1,\dots, d$, $|x|\in \{N-R+1,\dots, N+R\}$, and $w\in J^{n-d-1}_{M+R}$, let $h_{i,x,w}$ be a Lipschitz function such that 
    $$h_{i,x,w}(v_1,\dots,v_d) = q(w,v_i-x), \quad \forall (v_1,\dots,v_d)\in I_{N+2M}^d,$$
     and $\supp h_{i,x,w}\subset I_{N+3M}^d$. We can take these functions to have Lipschitz constant at most $2\sqrt{n}$.

    Suppose that $v=(v_1,\dots,v_n)\in K (Y_i)$. Then 
    $$K^{-1}(v) = (v_1-v_n,\dots,v_d-v_n,v_{d+1},\dots,v_n)\mbox{ and }|v_i-v_n|\in (N-R,N+R].$$ 
    Therefore, if $x=v_i-v_n$ and $w=(v_{d+1},\dots,v_{n-1})$, then $h_{i,x,w}(v_1,\dots, v_d) = v$. That is,
    $$K(Y_i)\subset \bigcup_{|x|\in \{N-R+1,\dots, N+R\}} \bigcup_{w\in J^{n-d-1}_{M+R}} \Lambda_{h_{i,x,w}}.$$
    Therefore, the graphs of the $h_{i,x,w}$'s cover $\bigcup_{i=1}^d K(Y_i)$.
    There are $d\cdot 2R\cdot (2M+2R+1)^{n-d-1}\lesssim M^{n-d-1}$ such functions. Thus, if we define $g_1,\dots, g_L$ to be a relabeling of the $k_z$'s and the $h_{i,x,w}$'s, we have $L\lesssim M^{n-d-1}$ and $X_R\subset \bigcup_{i=1}^L \Lambda_{g_i}$, as desired.
  \end{proof}

  To cover $\supp(\tilde{Q}\cap \partial[E])$ by Lipschitz graphs, we aim below to decompose $\supp W$ into simplices and apply Lemma~\ref{lem:extend-graphs}. Let $\sigma$ be a $\Z^n$--invariant triangulation of $\R^n$ such that $F$ is a subcomplex of $\sigma$, each simplex of $\sigma$ has diameter at most $\frac{1}{2}$, and $W$ is a simplicial chain in $\sigma$. We enumerate the $d$--simplices of $\supp W$ as $\delta_1,\dots,\delta_\ell$. The tangent planes to the $\delta_j$'s are intersections of tangent planes of $\partial F$ and tangent planes of $\tilde{Q}$, so the $\delta_j$'s are transverse to $P_0^\perp$; we orient each $\delta_j$ positively. Then we can write
  \begin{equation}\label{eq:tildeQ-decomp}
    \tilde{Q}\cap \partial [E] = \sum_{v \in X_R} \sum_{j=1}^\ell a_{v,j} [v\delta_j]
  \end{equation}
  for some coefficients $a_{v,j}\in \R$, where each coefficient $a_{v,j}$ is a sum of coefficients of $W$. In particular,
  $$\supp(\tilde{Q}\cap \partial [E])\subset \bigcup_{v\in X_R}\bigcup_{j=1}^\ell v\delta_j$$
  and
  \begin{equation}\label{eq:bdry-size}
    \mass(\tilde{Q}\cap \partial [E]) \lesssim \#X_R \sum_{j=1}^\ell\mathcal H^d(\delta_j) \lesssim N^dM^{n-d-1}.
  \end{equation}
  By applying Lemma~\ref{lem:extend-graphs} to the Lipschitz functions constructed in Lemma~\ref{lem:cover-XR}, we obtain Lipschitz functions $g_{i,j}\from [-N-3M,N+3M]^d\to [-2M,2M]^{n-d}$ such that 
  $$\bigcup_{v\in X_R}\bigcup_{j=1}^\ell v\delta_j \subset \bigcup_{i=1}^L\bigcup_{j=1}^\ell\Lambda_{g_{i,j}}$$
  and $\Lip(g_{i,j}) \lesssim 1$. 
  
  Let $c$ be the largest coefficient of $\tilde{Q}\cap \partial [E]$, and let 
  $$\Upsilon = \sum_{v\in J^{n-d}_{M+3}} [\Lambda_{f_v}] + c \sum_{i,j} [\Lambda_{g_{i,j}}].$$
  Then $\supp \Upsilon$ contains the support of $\tilde{T}_{P_0} \cap E$ and the support of $\tilde{Q}\cap \partial [E]$. Let $B_0:=A_0+\Upsilon$. The simplices of $\Upsilon$ cancel out any negatively oriented simplices of $\tilde{T}_{P_0} \cap E$ and $\tilde{Q}\cap \partial [E]$, so $A_0+\Upsilon$ is positively oriented. Furthermore, 
  $$\partial B_0 = \partial(A_0 + \Upsilon) = \partial \Upsilon = (\#J^{n-d}_{M+3} + c L \ell) [I_{N+3M}^d]$$
  is a multiple of $[I_{N+3M}^d]$. Let 
  \begin{equation}\label{eq:def-K}
    \alpha=(\#J^{n-d}_{M+3} + c L \ell)= (2M)^{n-d} + O(\epsilon M^{n-d}).
  \end{equation}

  It remains to compute the weighted Gaussian image of $B_0$. 
  By \eqref{eq:a0-break}, $B_0 = B_1 + B_2 + B_3$ where
  $$B_1 = \sum_{i=1}^k m_i \tilde{T}_{P_i}\rstr E,$$
  $$B_2 = -\tilde{T}_{P_0}\rstr E + \sum_{v\in J^{n-d}_{M+3}} [\Lambda_{f_v}],$$
  $$B_3 = \tilde{Q}\cap \partial[E] + \sum_{i,j} [\Lambda_{g_{i,j}}].$$
  Let 
  \begin{equation}\label{eq:def-V}
    V = \#\Xi=(2N+1)^d(2M+1)^{n-d} = 2^n N^d M^{n-d} + O(\epsilon N^dM^{n-d}).
  \end{equation}
  Then $\gamma_{B_1} = V\mu$. We claim that $\mass(B_2) + \mass(B_3)\lesssim \epsilon V$.

  To bound $\mass(B_2)$, note that the functions $f_v\from I_{N+3M}^d\to \R^{n-d}$ are constructed so that $\sum_v [\Lambda_{f_v}]\rstr E = \tilde{T}_{P_0}\rstr E$. Therefore, $B_2=\sum_v [\Lambda_{f_v}]\rstr(\R^n\setminus E)$, and
  \begin{equation}\label{eq:massB2}
    \mass(B_2) = \sum_{v\in J^{n-d}_{M+3}} \mass [\Lambda_{f_v}\setminus E].
  \end{equation}
  We have $\alpha I_{N-3, M-3} \Subset E$, so when $v\in J_{M-3}^{n-d}$, the intersection $\Lambda_{f_v}\cap E$ is a translate of $I_{N-3}^d$. Therefore, $\Lambda_{f_v}\setminus E$ is the graph of a Lipschitz function over a domain of volume $(2N+6M)^d - (2N-6)^d\approx N^{d-1} M$, and $\vol^d(\Lambda_{f_v}\setminus E)\approx N^{d-1}M$. Therefore,
  $$\mass(B_2)\lesssim \#J_{M-3}^{n-d} \cdot N^{d-1} M + \#(J_{M+3}^{n-d}\setminus J_{M-3}^{n-d})N^d \approx N^{d-1}M^{n-d+1} + N^d M^{n-d-1},$$
  i.e., $\mass(B_2)\lesssim V (\frac{M}{N}+\frac{1}{M})\lesssim \epsilon V.$
  
  Finally, we bound $B_3$. The functions $g_{i,j}$ are uniformly Lipschitz and defined on $I^d_{N+3M}$, so 
  $$\mass \sum_{i=1}^L\sum_{j=1}^\ell [\Lambda_{g_{i,j}}]\lesssim N^d L \ell \lesssim N^d M^{n-d-1}.$$
  Thus, by \eqref{eq:bdry-size}, 
  $$\mass(B_3) \lesssim N^dM^{n-d-1} + N^d M^{n-d-1}\lesssim \frac{V}{M}\lesssim \epsilon V.$$

  Therefore, by Lemma~\ref{lem:dist-decrease},
  $$\|\gamma_{B_0} - V\mu\|_{\TV} = \|\gamma_{B_0} - \gamma_{B_1}\|_{\TV} \le \mass(B_2) + \mass(B_3) \lesssim V \Big(\frac{M}{N}+\frac{1}{M}\Big).$$
  When $N$ is sufficiently larger than $M$ and both are large enough, $\|\gamma_{B_0} - V\mu\|_{\TV}\lesssim V \epsilon$.
  Let $B:=(\rho_{2N+6M})_\sharp(\alpha^{-1} B_0)$. Then $\partial B$ is a unit cube in $P_0$ and
  $$\Big\|\gamma_{B} - \frac{V}{\alpha (2N+6M)^{d}}\mu\Big\|_{\TV} \lesssim \frac{V\epsilon}{\alpha (2N+6M)^{d}}.$$
  By \eqref{eq:def-K} and \eqref{eq:def-V}, 
  \begin{equation}
      \begin{split}
          \alpha (2N+6M)^{d} &= 2^n(M^{n-d}+O(\epsilon M^{n-d}))(N^d+O(\epsilon N^d)) \\
          &= 2^n M^{n-d} N^d + O(\epsilon M^{n-d}N^d)= V + O(\epsilon V),
      \end{split}
  \end{equation}
  so $\|\gamma_{B} - \mu\|_{\TV} \lesssim \epsilon,$
  as desired.
\subsection{Proof of Lemma~\ref{lem:rational-approx}} 
We first prove the lemma without the orientation conditions.
Without loss of generality, suppose $\mu$ is a probability measure. 
  The set of planes with rational slope is dense in $\OGr(d,n)$, so for any $\epsilon'>0$, there is another probability measure $\mu_0$ which is a sum of $\delta_P$'s such that $d_{\mathsf{W}}(\mu_0 ,\mu)< \epsilon'$. Then
  $$\left|\int \omega_P \ud \mu(P) - \int \omega_P \ud \mu_0(P)\right| \lesssim \epsilon'.$$

  Let $\{Q_i\}_{i=1}^m\subset \OGr(d,n)$ be such that the $\omega_{Q_i}$'s span $\wedge^d\R^n$. Hence there exists $\{b_i\}_{i=1}^m\subset \R$ such that
  $$ \sum_{i=1}^m b_i \omega_{Q_i} = \int \omega_P \ud \mu(P) - \int \omega_P \ud \mu_0(P).$$
  In particular we construct the positive Radon measure 
  $$\nu = \sum_{i=1}^m |b_i| \delta_{\text{sign}(b_i) Q_i}\in \mathcal{M}(\OGr(d,n))$$ such that $\nu(\OGr(d,n))\lesssim \epsilon'$ and
  $$\int \omega_P \ud \nu(P) = \int \omega_P \ud \mu(P) - \int \omega_P \ud \mu_0(P).$$
  When $\epsilon'$ is sufficiently small, $\mu' := \frac{\mu_0 + \nu}{(\mu_0 + \nu)(\OGr(d,n))}$ satisfies the desired conditions.
  
  Now suppose that $\int \omega_P \ud \mu(P) = v \omega_{P_0}$ and the support of $\mu$ is positively oriented with respect to $P_0$. In this case, the method above does not work because $\text{sign}(b_i) Q_i$ need not be positively oriented. Instead, we first need to show that $\omega_{P_0}$ is in the interior of the convex hull of the positively oriented $d$-vectors with respect to $\omega_{P_0}$.

  To this aim, let $e_1,\dots, e_n$ be the coordinate basis of $\R^n$; without loss of generality, we suppose that $P_0 = \langle e_1,\dots, e_d\rangle$. Recalling Definition \ref{def:matrix-polycon}, any positively oriented plane has a basis of the form $w_1(X),\dots, w_d(X)$ for some matrix $X\in M_{n-d,d}(\R)$. Denoting $W(X) = \bigwedge M(X)$, as defined in Definition \ref{def:matrix-polycon},
  let $C$ be the convex hull of $W(M_{n-d,d}(\R))$. We claim that $\omega_{P_0}$ lies in the interior of $C$.
  
  The coordinates of $W(X)$ with respect to the standard basis of $\wedge^d\R^n$ are given by the minors of $X$. For each multi-index $I=(i_1,\dots, i_d)$ with $1\le i_1<\dots<i_d\le n$, let $v_I=e_{i_1}\wedge\dots \wedge e_{i_d}$ and let $\delta(I)$ be the number of $i_j$'s such that $i_j>d$. Then the $v_{(1,\dots, d)}$--coordinate of $W(X)$ is $1$ and the $v_I$--coordinate of $W(X)$ is given by a $\delta(I)\times \delta(I)$ minor of $X$. If the coefficients of $X$ are zero except in that minor, then the only nonzero coordinates of $W(X)$ are the $v_{(1,\dots, d)}$--coordinate, the $v_I$--coordinate and coordinates corresponding to smaller minors of that minor.

  It follows that $v_{(1,\dots, d)} \pm v_I\in C$ whenever $\delta(I)=1$. One can show by induction on $\delta(I)$ that there is an $r>0$ such that $v_{(1,\dots, d)} \pm r v_I \in C$ for all $I\ne (1,\dots,d)$. Thus $v_{(1,\dots, d)}=\omega_{P_0}$ lies in the interior of $C$.
  
  Now we conclude the proof of the lemma. Let $\eta>0$, let $\mu_0$ be a sum of point measures supported on positively-oriented planes such that $\mu_0(\OGr(d,n)) = \mu(\OGr(d,n))$ and $d_{\mathsf{W}}(\mu_0 ,\mu)< \eta^2$. Let $\omega = \int \omega_P \ud \mu_0(P)$. Then
  $|v\omega_{P_0} - \omega| \lesssim \eta^2$. Let
  $$\omega' = (1+\eta) v \omega_{P_0} - \omega.$$
  Then 
  $$|\omega' - \eta v \omega_{P_0}| = |v\omega_{P_0} - \omega|\lesssim \eta^2$$
  and thus $|v^{-1} \eta^{-1} \omega' - \omega_{P_0}|\lesssim v^{-1}\eta$.
  Since $\omega_{P_0}$ is in the interior of $C$, we can choose $\eta$ small enough that $v^{-1} \eta^{-1} \omega'\in C$ as well. Then there are positively-oriented $\{Q_i\}_{i=1}^m\subset \OGr(d,n)$ and coefficients $\{b_i\}_{i=1}^m\subset [0,1]$ such that $v^{-1} \eta^{-1} \omega' = \sum_i b_i\omega_{Q_i}$. We let $\nu = \sum_i b_i \eta v \delta_{Q_i}$ and we observe that
  $$\int \omega_P \ud(\mu_0+\nu)(P) = \omega + \omega' = (1+\eta) v \omega_{P_0},$$
  so, when $\eta$ is sufficiently small,
  $$\mu' := \frac{\mu_0 + \nu}{(\mu_0 + \nu)(\OGr(d,n))}$$
  satisfies the desired conditions.

\subsection{Proof of 
Theorem \ref{thm:poly-multi}}\label{sec:poly-multi} 

Given $\psi\in C^0(\R^{(n-d)\times d},[0,\infty))$, we prove that the energy $F_\psi$ is elliptic on Lipschitz multivalued functions if and only if $\psi$ is polyconvex.

Let $\bigwedge M \from \R^{(n-d)\times d}\to \bigwedge^d\R^n$ be as in Definition~\ref{def:matrix-polycon}. Let $B^d$ be the unit ball in $\R^d$. Let $\alpha\from B^d\to \R^{n-d}$ be a linear function whose graph is the plane $P_\alpha$ and let $h\from B^d\to \mathcal{A}_Q(\mathbb{R}^{n-d})$ be a $Q$--valued Lipschitz function such that $h(x)=Q\delta_{\alpha(x)}$ for every $x\in \partial B^d$. Then the multigraph $\Lambda_h$ provides the integral current $[\Lambda_h]$, see \cite[Definition 1.10]{DLSp2}, such that $\partial [\Lambda_h]=Q\partial [\alpha(B^d)]$. So by Stokes' Theorem, 
$$\int \omega_{P} \ud \gamma_{[\Lambda_h]}(P) = Q \cH^d(\alpha(B^d)) \omega_{P_\alpha}.$$

If $\psi\from \R^{(n-d)\times d}\to [0,\infty)$ is polyconvex, then there is a 1--homogeneous convex function $\Psi\from (\bigwedge^d\R^n)^+ \to [0,\infty)$ such that $\psi=\Psi\circ \bigwedge M$. Then, by Jensen's inequality,
$$F_\psi(h) = F_\Psi([\Lambda_h]) = \int \Psi(\omega_P) \ud \gamma_{[\Lambda_h]}(P)\ge Q \cH^d(\alpha(B^d)) \Psi(\omega_{P_\alpha}) = Q F_\psi(\alpha),$$
i.e., $\psi$ is elliptic for multivalued Lipschitz functions.

Conversely, suppose that $\psi\from \R^{(n-d)\times d}\to [0,\infty)$
is continuous but not polyconvex. Then there exist $X_0, X_1,\dots X_k\in \R^{(n-d)\times d}$ and coefficients $m_1,\dots, m_k$ such that 
  $$\sum^{k}_{i=1} m_i \bigwedge M(X_i) = \bigwedge M(X_0)$$
  and 
  $$\sum^{k}_{i=1} m_i \psi(X_i) < \psi(X_0).$$
  We assume without loss of generality that $X_0=0$.
  Let $P_i$ be the span of the column vectors of $M(X_i)$ and let  
  $$\mu:=\sum^{k}_{i=1} m_i \big \|\bigwedge M(X_i)\big \| \delta_{P_i}\in \cM(\OGr(d,n)).$$

For every $\epsilon>0$, the third part of Theorem~\ref{thm:construct} then provides a $Q$--valued function $u_\epsilon\from B^d\to \mathcal{A}_Q(\mathbb{R}^{n-d})$ such that
 \begin{equation}\label{closeness}
    \gamma_{[\Lambda_{u_\epsilon}]} \rightharpoonup^* \mu \; \mbox{ as } \; \epsilon \to 0\qquad \text{and} \qquad u|_{ \partial B^d}=0.  
  \end{equation}
  The continuity of $\psi$ implies that $F(u_\epsilon)< F(0)$, provided $\epsilon$ is chosen small enough. 
    Thus $\psi$ is not elliptic for Lipschitz multigraphs.

\section{The atomic condition implies strict polyconvexity}\label{sec:atomic}
Finally, we use Theorem \ref{thm:construct} to study the atomic condition and the equivalent condition (BC) defined in Definition \ref{def:BC}. Throughout this section we focus on even integrands, that is, we assume $\Psi\in C^1(\Gr(d,n),(0,\infty))$. We will prove Theorem~\ref{thm:main-atomic2}, which states that if $\Psi\in C^1(\Gr(d,n),(0,\infty))$ satisfies the atomic condition, then the even extension $\tilde \Psi$ of $\Psi$ to $\OGr(d,n)$ is strictly polyconvex. With a slight abuse of notation, we will denote $\Psi\in C^1(\Gr(d,n),(0,\infty))$ and its even extension by the same letter $\Psi$. The domain of the integrand will be clear from the context. 

The main tool in the proof is the following proposition.
\begin{prop}\label{stationarylemma}
  Let $S\in \bI^{d-1}(\R^n)$ and $\{A^j\}_{j\in \N}\subset \bR^d(\R^n)$ be a minimizing sequence for $\FE_F(S;\R)$, i.e. 
  $$\partial A^j=S \quad \forall j\in \N \qquad \mbox{ and } \qquad \lim_{j\to \infty}F(A^j)= \FE_F(S;\R).$$
  Let $V^j:=\cV(A^j)$ be the varifold canonically associated with $A^j$. Suppose that the $V^j$'s converge to a varifold $W$. Then $W$ is stationary away from $\supp(S)$, i.e., $\delta_\Psi[W](g)=0$ for every smooth vector field $g$ compactly supported in $\R^{n}\setminus \supp(S)$.
\end{prop}
\begin{proof}
  Suppose by way of contradiction that $g$ is a smooth vector field compactly supported in $\R^{n}\setminus \supp(S)$ such that
  $$\delta_\Psi[W](g)\leq -2C <0.$$
  If $\epsilon>0$ is small enough, then we can define $\varphi_t\from \R^n\to \R^n, t\in (-\epsilon,\epsilon)$ to be the flow of $g$, so that $\varphi_0=\id$, $\frac{\partial \varphi_t(x)}{\partial t}=g(\varphi_t(x))$ for all $|t|<\epsilon$ and all $x\in \R^n$, and $\varphi_t$ is a diffeomorphism for all $|t|<\epsilon$.

  Let $W_t:=(\varphi_t)_\sharp W$ and let $V^j_t=(\varphi_t)_\sharp V^j$, so that $$\left.\frac{d}{d t}\right|_{t = 0}F(W_t) = \delta_\Psi[W](g)\leq -2C.$$
  Since the functional $Z\mapsto\delta_\Psi[Z](g)$ is continuous with respect to $Z$, we may shrink $\epsilon$ so that 
  \begin{equation}
    \label{integro1}F(W_t)\leq F(W) - Ct \text{\qquad for }t\in [0,\epsilon).
  \end{equation}

  Let $B$ be a ball such that $g$ is compactly supported in $B$ and $\|W\|(\partial B)=0$. Hence $\lim_{j\to \infty}F(V^j,B) = F(W,B)$. Since $\varphi_t$ is the identity on a neighborhood of $\R^n\setminus B$, we have $\|W_t\|(\partial B) = 0$ for all $|t|<\epsilon$, 
  $$F(V^j) - F(V^j_t) = F(V^j,B) - F(V^j_t,B),$$
  and
  $$F(W) - F(W_t) = F(W,B) - F(W_t,B).$$

  Therefore, $\lim_{j\to \infty}F(V^j_t,B) = F(W_t,B)$ for all $|t|<\epsilon$ and
  $$\lim_{j\to \infty} F(V^j) - F(V^j_t) = \lim_{j\to \infty} F(V^j, B) - F(V^j_t,B) = F(W,B) - F(W_t,B) \stackrel{\eqref{integro1}}{\ge} Ct.$$ 
  Since $A^j$ is a minimizing sequence for $\FE_F(S;\R)$ and $F(A^j)=F(V^j)$, this implies that for all $|t|<\epsilon$
  \begin{equation}\label{contr}
       \lim_{j\to \infty} F(V^j_t) \le \lim_{j\to \infty} F(V^j) - Ct = \FE_F(S;\R) - Ct.
  \end{equation}
  Since $\varphi_s$ is a diffeomorphism of $\R^n\setminus \supp(S)$ into itself, the rectifiable current $A^j_t:=(\varphi_t)_\sharp A^j$ satisfies $\partial A^j_t=S$ and $F(A^j_t)=F(V^j_t)$. Thus, there are $j$ and $t$ such that $F(A^j_t)<\FE_F(S;\R)$, which is a contradiction. Therefore, $W$ is stationary away from $\supp(S)$.
\end{proof}

Thus, one can prove that $\Psi$ does not satisfy the atomic condition by finding a minimizing sequence for $\FE_F$ that converges to a varifold $W$ of the form $(\mathcal H^d\rstr\supp(D))\otimes \mu$, where $\supp \mu$ consists of more than one point, and by blowing up $W$ at a point in the relative interior of $D$.

We give two such constructions. First, when $\Psi$ is not strictly elliptic, we can construct a minimizing sequence by tiling. 
\begin{prop}\label{atco2}
  If $\Psi\in C^1(\Gr(d,n),(0,\infty))$ satisfies the atomic condition, then $\Psi$ is strictly elliptic for rectifiable currents over $\R$.
\end{prop}
\begin{proof}
  Suppose that $\Psi$ is not strictly elliptic for rectifiable currents. Then there exists $P_0 \in \OGr(d,n)$ such that if $[D]\in \bI^d(P_0)$ is the fundamental class of a unit cube $D\subset P_0$ and $S=\partial [D]$, then either $\FE_F(S;\R)<F([D])$ or $\FE_F(S;\R)=F([D])$ and there is a current $A\in \bR^d(\R^n)$ such that $\partial A=S$ and $F(A) = F([D])$, but $A\ne [D]$.

  If $\FE_F(S;\R)<F([D])$, let $A^j$ be a minimizing sequence for $\FE_F(S;\R)$ such that $F(A^j)<F([D])$. Otherwise, let $A^j=A$ for all $j$. 
  
  Since $\gamma_{A^j}(\OGr(d,n)) = \mass A^j \geq \mass [D]>0$ is bounded as $j\to \infty$, we may pass to a subsequence so that $\gamma_{A^j}$ converges weakly to a measure $\mu \neq 0$ on $\OGr(d,n)$. If $\FE_F(S;\R)<F([D])$, then 
  $$\int \Psi(P)\ud \mu(P) = \lim_{j\to \infty} \int \Psi(P)\ud \gamma_{A^j}(P) =\FE_F(S;\R) < F([D]) = \Psi(P_0),$$
  so, we have that $\supp p_*\mu\ne \{p(P_0)\}$. Otherwise, $\supp p_*\mu = \supp p_*\gamma_A \ne \{p(P_0)\}$, since $A\ne [D]$. In either case, $\supp p_*\mu \ne \{p(P_0)\}$.

  We will use the tiling argument in the proof of \cite[Theorem 8.8]{DeRosaKolasinski} to construct a minimizing sequence whose corresponding varifolds converge to $W:=(\cH^d\rstr D )\otimes p_*\mu$. 

  For a $d$--current $B$ and for $i\in \N$, let $\tau_i(B)$ be the current obtained by first rescaling $B$ by a factor of $\frac{1}{2^i}$, then tiling $2^{id}$ copies of the rescaled $B$ in an $2^i\times \dots \times 2^i$ grid. We arrange the grid so that if $\partial B = S$, then $\partial \tau_i(B)=S$ for all $i$. Then, for $i\ge 0$, we can write
  $$\tau_{i+1}(B) = \sum_{k=1}^{2^{d}} B_k,$$
  where each $B_k$ is a rescaling of $\tau_{i}(B)$ with $\gamma_{B_k} = 2^{-d}\gamma_{\tau_i(B)}$. It follows from Lemma \ref{subadditivity} that
  $$\gamma_{\tau_{i+1}(B)}\le \sum_{k=1}^{2^d} \gamma_{B_k} = \gamma_{\tau_i(B)}.$$
  Therefore, $F(B)\ge F(\tau_i(B))$ for all $i$ and $F(\tau_i(B))$ is nonincreasing in $i$. Likewise, $\mass(B)\ge \mass(\tau_i(B))$ and $\mass(\tau_i(B))$ is nonincreasing in $i$.
  
  Since $\gamma_{\tau_{i}(B)}$ is decreasing, it converges weakly to a measure $\mu_B$ on $\OGr(d,n)$. Furthermore, the varifold $\cV(\tau_{i}(B))$ associated to $\tau_i(B)$ satisfies
  $$\lim_{i\to \infty} \cV(\tau_{i}(B)) = (\cH^d\rstr D) \otimes p_*\mu_B.$$
  This is straightforward to show in the case that $B$ has compact support, and in the general case, it follows from approximating $B$ by compactly supported currents.
  
  We claim that $\mu_{A^j}$ converges to $\mu$. We have $\mu_{A^j}\le \gamma_{A^j}$ and
  $$\int \Psi(P) \ud \mu_{A^j}(P) = \lim_{i\to \infty} \int \Psi(P) \ud \gamma_{\tau_i(A^j)}(P) = \lim_{i\to \infty} F(\tau_i(A^j)) \ge \FE_F(S),$$
  so
  $$\int \Psi(P) \ud (\gamma_{A^j}-\mu_{A^j})(P) = F(A^j) - \int \Psi(P) \ud \mu_{A^j}(P) \le F(A^j) - \FE_F(S).$$
  Since $\OGr(d,n)$ is compact and $\Psi>0$, there is a $c>0$ such that
  $$\|\gamma_{A^j}-\mu_{A^j}\| \le c(F(A^j) - \FE_F(S)),$$
  and since $A^j$ is a minimizing sequence, 
  $$\lim_{j\to\infty} \gamma_{A^j}-\mu_{A^j} = 0.$$
  Therefore, $\lim_{j\to \infty} \mu_{A^j} = \lim_{j\to\infty} \gamma_{A^j} = \mu$.

  Thus, in the sense of varifolds
  $$\lim_{j\to \infty} \lim_{i\to \infty} \cV(\tau_{i}(A^j)) = \lim_{j\to\infty} (\cH^d\rstr D )\otimes p_*\mu_{A^j} = (\cH^d\rstr D )\otimes p_*\mu=W.$$
  By diagonalizing, there is a sequence $i_j\to \infty$ such that $\lim_{j\to \infty}\cV(\tau_{i_j}(A^j)) = W$, and since $F(\tau_{i_j}(A^j))\le F(A^j)$, then $\tau_{i_j}(A^j)$ is a minimizing sequence. By Proposition~\ref{stationarylemma}, $W=(\cH^d\rstr D) \otimes p_*\mu$ is a stationary varifold away from $\supp(S)$, hence the blowup $\tilde W:=(\cH^d\rstr p(P_0)) \otimes p_*\mu$ satisfies $\delta_\Psi[\tilde W]=0$, but $\supp p_*\mu \ne \{p(P_0)\}$, so $F$ does not satisfy the atomic condition.
\end{proof}

\begin{remark}
    We observe that Proposition \ref{atco2} should be compared with \cite[Theorem A and Theorem 8.8]{DeRosaKolasinski}. In \cite[Theorem A and Theorem 8.8]{DeRosaKolasinski} the atomic condition is shown to imply the strict ellipticity for rectifiable sets, while in Proposition \ref{atco2} we consider the strict ellipticity for rectifiable currents over $\R$. We have also  slightly simplified the proof; in \cite{DeRosaKolasinski}, one needed to solve the anisotropic Plateau problem, but in fact, Proposition \ref{stationarylemma} allows us to work directly on a minimizing sequence, avoiding this step. A similar simplification can be applied in \cite[Theorem A and Theorem 8.8]{DeRosaKolasinski}.
\end{remark}

Second, when $\Psi$ is not strictly polyconvex, we can use Theorem \ref{thm:construct} to construct a minimizing sequence.
\begin{proof}[Proof of Theorem \ref{thm:main-atomic2}]
  Suppose that $\Psi$ is not strictly polyconvex. Then there are an oriented $d$--plane $P_0\subset \R^n$ and a measure $0\neq \mu\in \mathcal{M}(\OGr(d,n))$ such that $\int \omega_P \ud \mu(P) = \omega_{P_0}$, $\mu\neq\delta_{P_0}$, and
  $$\int \Psi(P)\ud \mu(P)\leq \Psi(P_0).$$
  First we observe that $\supp p_*\mu \ne \{p(P_0)\}$, otherwise we would have $\supp \mu = \{P_0,-P_0\}$, which by $\int \omega_P \ud \mu(P) = \omega_{P_0}$ implies that $\mass \mu >1$, which in turn implies the following contradiction
  $$\int \Psi(P)\ud \mu(P)=\Psi (P_0) \mass \mu > \Psi(P_0).$$

  Let $D$ be a unit $d$--cube in $P_0$ and let $S=\partial [D]$. By the compactness of $\OGr(d,n)$, we may suppose that $\mu$ minimizes $\int \Psi(P)\ud \mu(P)$, i.e., 
  $$\int \Psi(P)\ud \mu(P) = \inf \left\{\int \Psi(P)\ud \nu(P) : \nu\in \mathcal{M}(\OGr(d,n)), \int \omega_P \ud \nu(P) = \omega_{P_0}\right\},$$
  and by Theorem~\ref{thm:inf-formula}, we have $\FE_F(S;\R) = \int \Psi(P)\ud \mu(P)$.

  By Theorem \ref{thm:construct}, there is a sequence $\{A^j\}_{j\in \N}\subset \bR^d(\R^n)$ such that
  $$\partial A^j=S \quad \forall j\in \N \qquad \mbox{ and } \qquad \gamma_{A^j} \rightharpoonup^* \mu \; \mbox{ as } \; j \to \infty,$$
  so that $\{A^j\}_{j\in \N}$ is a minimizing sequence for $\FE_F(S;\R)$. 
  Furthermore, the rectifiable varifolds $\cV(A^j)$ associated to the $A^j$'s converge to $W:= (\mathcal H^d\rstr D)\otimes p_*\mu$ as $j\to \infty$. 
  
  By Proposition~\ref{stationarylemma}, $W$ is stationary away from $\supp(S)$. As in the proof of Proposition~\ref{atco2}, the blowup $\tilde W:=(\mathcal H^d \rstr p(P_0)) \otimes p_*\mu$ satisfies $\delta_\Psi[\tilde W]=0$, so, since $\supp p_*\mu \ne \{p(P_0)\}$, $\Psi$ violates (BC).
\end{proof}

{\bf Data Availability Statement.} No data was used for the research described in the
article.

\bibliographystyle{plain}

\end{document}